\documentclass[final,onefignum,onetabnum]{siamart220329}

\usepackage{braket,amsfonts,amsmath,amssymb}
\usepackage[title]{appendix}
\usepackage{array}
\usepackage{enumerate}
\usepackage{multirow}
\usepackage[caption=false]{subfig}
\usepackage{pgfplots}

\newsiamremark{remark}{Remark}
\usepackage{algorithmic}

\usepackage{graphicx,epstopdf}

\Crefname{ALC@unique}{Line}{Lines}

\usepackage{amsopn}

\usepackage{xspace}
\usepackage{bold-extra}
\usepackage[most]{tcolorbox}

\colorlet{texcscolor}{blue!50!black}
\colorlet{texemcolor}{red!70!black}
\colorlet{texpreamble}{red!70!black}
\colorlet{codebackground}{black!25!white!25}

\lstdefinestyle{siamlatex}{%
  style=tcblatex,
  texcsstyle=*\color{texcscolor},
  texcsstyle=[2]\color{texemcolor},
  keywordstyle=[2]\color{texemcolor},
  moretexcs={cref,Cref,maketitle,mathcal,text,headers,email,url},
}

\tcbset{%
  colframe=black!75!white!75,
  coltitle=white,
  colback=codebackground, % bottom/left side
  colbacklower=white, % top/right side
  fonttitle=\bfseries,
  arc=0pt,outer arc=0pt,
  top=1pt,bottom=1pt,left=1mm,right=1mm,middle=1mm,boxsep=1mm,
  leftrule=0.3mm,rightrule=0.3mm,toprule=0.3mm,bottomrule=0.3mm,
  listing options={style=siamlatex}
}

\newtcblisting[use counter=example]{example}[2][]{%
  title={Example~\thetcbcounter: #2},#1}

\newtcbinputlisting[use counter=example]{\examplefile}[3][]{%
  title={Example~\thetcbcounter: #2},listing file={#3},#1}

\DeclareTotalTCBox{\code}{ v O{} }
{ %fontupper=\ttfamily\color{texemcolor},
  fontupper=\ttfamily\color{black},
  nobeforeafter,
  tcbox raise base,
  colback=codebackground,colframe=white,
  top=0pt,bottom=0pt,left=0mm,right=0mm,
  leftrule=0pt,rightrule=0pt,toprule=0mm,bottomrule=0mm,
  boxsep=0.5mm,
  #2}{#1}

\patchcmd\newpage{\vfil}{}{}{}
\title{Nehari manifold optimization and its application for finding unstable solutions of semilinear elliptic PDEs\thanks{\funding{This work was supported by the NSFC grant 12171148. The work of Z.~Chen and Z.~Xie was also supported by the Major Program of Xiangjiang Laboratory, China (No. 22XJ01013). The work of W.~Liu was also supported by the NSFC grant 12101252 and the NUDT grant 202402-YJRC-XX-002. The work of W.~Yi was also supported by the NSFC grant 12471374, Hunan Provincial Natural Science Foundation (No.~2025JJ40001), and the Key Scientific Research Project of the Education Department of Hunan Province (No.~23A0034).}}}

\author{Zhaoxing Chen\thanks{MOE-LCSM,  School of Mathematics and Statistics, Institute of Interdisciplinary Studies, Hunan Normal University, Changsha 410081, China (\email{zhaoxingchen@hunnu.edu.cn}).}
\and Wei Liu\thanks{Department of Mathematics, National University of Defense Technology, Changsha 410073, China (\email{wl@nudt.edu.cn}).}
\and Ziqing Xie\thanks{Corresponding author. MOE-LCSM, School of Mathematics and Statistics, Institute of Interdisciplinary Studies, Hunan Normal University, Changsha 410081, China (\email{ziqingxie@hunnu.edu.cn}).}
\and Wenfan Yi\thanks{Hunan Provincial Key Laboratory of Intelligent Information Processing and Applied Mathematics, School of Mathematics, Hunan University, Changsha, 410082, China (\email{wfyi@hnu.edu.cn}).}}

\headers{Nehari manifold optimization}{Zhaoxing~Chen, Wei~Liu, Ziqing~Xie, and Wenfan~Yi}
\ifpdf
\hypersetup{ pdftitle={Nehari manifold optimization} }
\fi

\begin{document}
\maketitle
%% ------------------------------------------------------------------
%% ABSTRACT
%% ------------------------------------------------------------------
% \begin{tcbverbatimwrite}{tmp_\jobname_abstract.tex}
\begin{abstract}
A {\em Nehari manifold optimization method} (NMOM) is introduced for finding 1-saddles, i.e., saddle points with the Morse index equal to one, of a generic nonlinear functional in Hilbert spaces. Actually, it is based on the variational characterization that 1-saddles of this functional are local minimizers of the same functional restricted on the associated Nehari manifold. The framework contains two important ingredients: one is the retraction mapping to make the iterative points always lie on the Nehari manifold; the other is the tangential search direction to decrease the functional with suitable step-size search rules. Particularly, the global convergence is rigorously established by virtue of some crucial analysis techniques (including a weak convergence method) that overcome difficulties in the infinite-dimensional setting. In practice, combining with an easy-to-implement Nehari retraction and the negative Riemannian gradient direction, the NMOM is successfully applied to compute the unstable ground-state solutions of a class of typical semilinear elliptic PDEs, such as the stationary nonlinear Schr\"odinger equation and the H\'enon equation. In particular, the symmetry-breaking phenomenon of the ground states of the H\'enon equation is explored numerically in 1D and 2D with interesting numerical findings on the critical value of the symmetry-breaking reported. 
\end{abstract}


% \end{tcbverbatimwrite}
% \input{tmp_\jobname_abstract.tex}
%% ------------------------------------------------------------------
%% END HEADER
%% ------------------------------------------------------------------
\section{Introduction}\label{section 1}
In this paper, we consider saddle points of a continuously Fr\'echet-differentiable functional $E$ which is nonlinear and defined in a real Hilbert space $H$. A saddle point $u^*$ of $E$ refers to a critical point but not a local extremizer (minimizer and maximizer) of $E$, i.e., the Fr\'echet derivative $E'(u^*)=0$ and any neighborhood of $u^*$ in $H$ contains points $w$ and $v$ s.t. $E(w)<E(u^*)<E(v)$. In physics, chemistry and materials science, etc., saddle points appearing as unstable equilibria or transient excited states of the energy surface, have broad applications and gain more and more interests \cite{2013BaoBEC,EZ2011-GAD,LMM,LXY-Goldstein,LXY-NWPLMM,LXY2023JCP,E2014Science,PWZhang2021PNAS}. By Morse theory, if the second-order Fr\'echet derivative $E''(u^*)$ at a critical point $u^*$ exists, the instability of $u^*$ can be depicted quantitatively by the Morse index (MI), which is defined as the maximal dimension of negative-definite subspaces of $E''(u^*)$. To be precise, let $H=H^{-}\oplus H^0\oplus H^+$, where $H^-$, $H^0$ and $H^+$ are respectively the maximum negative-definite, null and maximum positive-definite orthogonal subspaces of the linear operator $E''(u^*)$ in $H$ with $\dim(H^0)<\infty$, then the Morse index of $u^*$ is defined as the dimension of $H^-$~\cite{1994Infinite}, denoted as $\operatorname{MI}(u^*)=\dim(H^-)$. For a nondegenerate critical point $u^*$ (i.e., $E''(u^*)$ is invertible), if its $\operatorname{MI}=0$, then it is a local minimizer which is stable; and if its $\operatorname{MI}>0$, then it is a saddle point which is unstable. Specially, a saddle point with $\mathrm{MI}=k$ ($k\in\mathbb{N}^+$) is called a {\em $k$-saddle} for simplicity. 
Due to the nonlinearity and nonconvexity of the functional $E$, as well as the instability and multiplicity of saddle points, there are great challenges on theoretical analysis and numerical computation for saddle points.

In the critical point theory, the Nehari manifold is a central concept for addressing the existence issue of critical points~\cite{1960On,1961Characteristic,Struwe2000,2010Themethod}. For the functional $E$, the associated Nehari manifold $\mathcal{N}$ is defined as 
\begin{align}\label{Nehari}
 \mathcal{N}=\left\{u\in H\backslash\{0\}: \langle E'(u),u\rangle=0\right\},
\end{align}
which clearly contains all nontrivial (non-zero) critical points of the functional $E$ in $H$. Hereafter, $\langle\cdot,\cdot\rangle$ denotes the duality pairing between $H$ and its dual space $H^*$. It is worthwhile to point out that, when $E$ is the energy functional corresponding to a certain nonlinear ODE/PDE, the global minimizer on the Nehari manifold $\mathcal{N}$ exists and is exactly a 1-saddle of the energy functional $E$ in $H$ \cite{1960On,1961Characteristic,Ni1989}. Inspired by the standard assumptions of the theoretical Nehari method in the critical point theory \cite{0Nonlinear,2010Themethod}, we consider the following basic assumptions to own a nice geometric setting of the Nehari manifold $\mathcal{N}$ for a generic functional $E$: 
\begin{itemize}
    \item[(A1)]  $E \in C^2(H,\mathbb{R})$ and for each $v \in H \backslash \{0\}$, the function $\Phi(s) = E(sv)$ w.r.t. $s\in(0,\infty)$ admits a unique critical point $s_v>0$; 
    \item[(A2)] $\langle E''(u)u,u \rangle <0$ for all $u \in \mathcal{N}$;
    \item[(A3)] there exists a constant $\sigma_0>0$ s.t. $\|u\|_H \geq \sigma_0$ for all $u\in\mathcal{N}$.
\end{itemize}
In fact, the assumption (A1) states that the ray $\{sv:s>0\}$ along each direction $v \in H \backslash \{0\}$ intersects $\mathcal{N}$ at a unique point $s_vv$ and $\mathcal{N}=\{s_vv: v \in H \backslash \{0\}, s_v > 0\}$. Moreover, with the assumption (A2), $s_v$ is actually a strict global maximum point of $\Phi(s) = E(sv)$ in $(0,\infty)$. Thus $E$ is bounded from below on $\mathcal{N}$ and $E(u) \geq E(0), \; \forall \; u \in \mathcal{N}$. In addition, (A3) states that $\mathcal{N}$ is bounded away from 0 and then $\mathcal{N}$ is a closed subset of $H$. As we will see later in Section~\ref{section 2}, under the assumptions (A1)-(A3), the Nehari manifold $\mathcal{N}$ \eqref{Nehari} for the functional $E$ is indeed a $C^1$ Riemannian submanifold of $H$, and the local minimizers of $E$ on $\mathcal{N}$ can characterize all (nondegenerate) 1-saddles of $E$ in $H$. This motivates us to consider the numerical computation for 1-saddles of a generic functional $E$ in $H$ via the minimization on the associated Nehari manifold $\mathcal{N}$ \eqref{Nehari}, i.e., 
\begin{align}\label{eq:minEonN}
    \min_{u\in \mathcal{N}}E(u),
\end{align}
which will be the main concern in this paper. 
In addition, 1-saddle is of great significance in some applications such as the study of rare events in which a 1-saddle represents the transition state between metastable states on a complex energy landscape \cite{CM1981JChemPhys,EZ2011-GAD,E2014Science,PWZhang2021PNAS,ZDZ2016SISC}.

In the literature, there is another popular way to characterize saddle points of the functional $E$ by minimax problems~\cite{LMM,LMMConvergence2002,Rabinowitz1986}, leading to the development of minimax-type numerical algorithms for finding saddle points~\cite{2000Algorithms,1993MPA,1999HLA,LMM,LMMConvergence2002,XieYuanZhou2012}. Particularly, Li and Zhou established a local minimax principle and designed a local minimax method (LMM) for computing multiple saddle points~\cite{LMM,LMMConvergence2002}. Note that for the case of 1-saddles, the LMMs can be reduced to consider the two-level minimax problem~\cite{LMM,LMMConvergence2002,XieYuanZhou2012}
\begin{align}\label{minimax characters}
    \min_{v \in S_H} \max_{u\in\{sv:s>0\}} E(u),
\end{align}
where $S_H = \{v\in H: \|v\|_H =1\}$ denotes the unit sphere in $H$. In fact, the two-level minimax problem \eqref{minimax characters} can be reformulated as $\min_{u\in \mathcal{M}} E(u)$, where $\mathcal{M}=\{s_vv: E(s_v v)= \max_{s>0} E(sv),v\in S_H\}$, called the `solution submanifold' \cite{LMM,LMMConvergence2002,XieYuanZhou2012}, is the set of all solutions to the local maximization problem at the first level. If the functional $E$ owns a unique nontrivial critical point in each direction and the critical point is maximum point, the `solution submanifold' $\mathcal{M}$ is identical to the Nehari manifold $\mathcal{N}$ \eqref{Nehari} under the assumptions (A1)-(A3). Hence, \eqref{minimax characters} provides a two-level local minimax characterization for \eqref{eq:minEonN}. By defining a composite functional $f(v) = E(p(v))$ with $p(v) = s_v v\in\mathcal{M}$ the so-called local peak selection in the LMMs \cite{LMM,LMMConvergence2002,XieYuanZhou2012}, a natural idea to address the minimax problem \eqref{minimax characters} is to consider the minimization problem $\min_{v \in S_{H}} f(v)$ on $S_{H}$.   However, the objective composite functional $f(v)$ is a little intricate.  LMMs adapt a two-level nested iterative scheme to numerically implement \eqref{minimax characters}, where a standard unconstrained optimization algorithm is applied at the first level and a normalized steepest descent algorithm on $S_{H}$ is designed to update $v$ at the second level.
While the minimization formulation \eqref{eq:minEonN} on the Nehari manifold is both mathematically and algorithmically straightforward to find 1-saddles of the functional $E$ in $H$, providing a clear pathway for the implementation and analysis. To our knowledge, no numerical methods have been developed to find saddle points by directly solving~\eqref{eq:minEonN} and utilizing geometric properties of the associated Nehari manifold. There also exist  diverse numerical methods for finding saddle points or multiple solutions in the literature,  whose motivations are quite different from our approach. Consequently, discussions on these methods are skipped and the details refer  to~\cite{CM1981JChemPhys,2004Search,EZ2011-GAD,FBF2015SISC,GLZ2015SINUM,LWL2023JSC,LWZ2017JSC,LXY2023JCP,LYZ-ActionGS,Wang2022JCP,XYZ2015JCAM,2008Yang,ZDZ2016SISC} and the references therein.

    Intuitively, the task of numerically solving \eqref{eq:minEonN} is naturally an optimization problem on a Riemannian manifold which is generally infinite-dimensional. Riemannian optimization has been well studied for finite-dimensional manifolds \cite{2008Optimization,boumal2023intromanifolds,HLWY2020JORSC}. Its basic idea is to introduce suitable retraction mappings to construct iterative sequences along the manifold by fully utilizing the geometry properties of the underlying manifold. For infinite-dimensional manifolds, due to the lack of compactness, there are great challenges to construct and analyze Riemannian optimization algorithms. Several numerical researches have been carried out for some special optimization problems on infinite-dimensional manifolds, such as the Kohn–Sham energy minimization on the infinite-dimensional Stiefel manifold (defined by $L^2$-orthogonality conditions) for the electronic structure computation \cite{2021Energy-adaptive} and the Gross-Pitaevskii energy minimization on the infinite-dimensional $L^2$-spherical manifold (or its submanifold) for Bose-Einstein condensates \cite{2017Tang,2017Computation,2020GroundBEC}. To our knowledge, few studies have focused on the convergence analysis for infinite-dimensional Riemannian optimization. It  is  worthwihile to point out that,  in~\cite{Ring2012}, under certain conditions for a (locally) strictly convex functional, the convergence analysis for standard Riemannian optimization methods has been provided in the infinite-dimensional cases. For the Nehari manifold $\mathcal{N}$ \eqref{Nehari} associated to the functional $E$ considered in this paper, the convexity of the functional $E$ cannot be guaranteed, so that the approach in~\cite{Ring2012} cannot be directly applied to our work.

In this paper, we focus on developing and analyzing Riemannian optimization approaches for solving the minimization problem \eqref{eq:minEonN} to obtain  1-saddles of the functional $E$ under the infinite-dimensional setting. 
The main contributions of this paper are three aspects. (i) For a generic functional $E$, the basic geometric settings of the Nehari manifold $\mathcal{N}$ is presented and the local minimization characterization on $\mathcal{N}$ \eqref{Nehari} for 1-saddles is established. (ii) A general framework of Riemannian optimization algorithms on the Nehari manifold $\mathcal{N}$ \eqref{Nehari}, named as the {\em Nehari manifold optimization method} (NMOM), is introduced for solving the problem~\eqref{eq:minEonN} with the flexibility on the selection of the {\em retraction}, the tangential {\em descent direction} and the {\em step-size search rule}. 
Specifically, a practical {\em Nehari retraction} is proposed by making full use of the geometry of the Nehari manifold $\mathcal{N}$ \eqref{Nehari}. 
(iii) The global convergence of the proposed NMOM is established based on a nonmonotone step-size search rule for general (strongly) gradient-related descent directions. In our analysis, a delicate weak convergence technique is developed, and an easily verifiable condition for the retraction weaker than the Lipschitz continuity is proposed to address the challenges in the infinite-dimensional setting. 

This paper is organized as follows.  Section~\ref{section 2} establishes the local minimization characterization for 1-saddles. Based on this, the algorithm framework of the NMOM is then introduced in Section~\ref{section 3}. The global convergence of the algorithm is provided in Section~\ref{sec:Convergence}. Section~\ref{Numerical experiments} presents some applications of the NMOM to compute unstable ground states for a class of semilinear elliptic PDEs. In Section~\ref{sec:symmetry-breaking}, the symmetry-breaking phenomenon of the H\'enon equation is explored numerically. Finally, some conclusions and discussions are drawn in Section~\ref{section:conclusion}.

\vspace{.5ex}
{\bf Notations.}
Throughout this paper, denote $(\cdot,\cdot)_H$ by the inner product in $H$ and $\|\cdot\|_H = (\cdot,\cdot)_H^{1/2}$ the induced $H$-norm. The $H$-gradient of a $C^1$ functional $F$ at $u \in H$, denoted by $\nabla F(u)$, is the Riesz representer of the Fr\'echet derivative $F'(u)$ in $H$, i.e.,
\begin{align}\label{eq:Hgrad-def}
(\nabla F(u), \phi)_H = \langle F'(u), \phi\rangle, \quad \forall\; \phi \in H.
\end{align}
By the Riesz representation theorem, \!$\nabla F(u)$ is well-defined and $\|\nabla F(u)\|_H \!=\! \|F'(u)\|_{H^*}$.

\section{Variational characterization of 1-saddles via Nehari manifold}\label{section 2}
In this section, under the assumptions (A1)-(A3), for  a generic functional $E:H\to\mathbb{R}$, we investigate the geometry of the associated Nehari manifold $\mathcal{N}$ \eqref{Nehari}. The variational characterization for 1-saddles of $E$ in $H$ is established via the minimization on the Nehari manifold $\mathcal{N}$ \eqref{Nehari}, and the existence of the ground-state critical point (i.e., the global minimizer on $\mathcal{N}$ \eqref{Nehari}) is also provided. 

Introducing the functional $G(u) = \langle E'(u),u\rangle$ for all $u\in H $, the Nehari manifold $\mathcal{N}$ \eqref{Nehari} can be expressed as a level set of $G$, i.e.,
\begin{align}
\mathcal{N} = \{u \in H\backslash\{0\}: G(u)= 0\}.
\end{align}
Under the assumptions (A1)-(A3), the functional $G$ is of $C^1$ and satisfies $\langle G'(u),u\rangle = \langle E''(u)u,u \rangle <0$, $\forall \;u\in \mathcal{N}$. Hence, $G'(u)\neq0$, $\forall \;u\in\mathcal{N}$. Recall that a closed subset $ \mathcal{M} \subset H $ is closed $C^m$ submanifold ($m\geq1$) of $H$ if the set of charts in $ \mathcal{M}$ which are restrictions of charts for $H$ form an atlas for $\mathcal{M}$ of class $C^m$ \cite{1995differential,Palais1963}. It can be verified that the Nehari manifold $\mathcal{N}$ is a $C^1$ submanifold, as stated in the following lemma, which is essentially known in the nonlinear analysis~\cite{0Nonlinear}. We give the proof here for the completeness of the paper and the readers' convenience. 

\begin{lemma}\label{lem:Nehari-submanifold}
Under the assumptions (A1)-(A3), the Nehari manifold $\mathcal{N}$ is a closed $C^1$ submanifold of $H$.
\end{lemma}

\begin{proof} 
Under the assumption (A3), the Nehari manifold $\mathcal{N}$ is closed in $H$. 
For every $u\in\mathcal{N}$, referring to~\cite[Section~6.3]{0Nonlinear}, denote $X_u=\{\xi\in H\, :\, (\nabla G(u),\xi)_H=0 \}$ and define 
\[\psi(v)= v - u - \langle
G'(u), v-u\rangle\frac{\nabla G(u)}{\|\nabla G(u)\|_H^2}  + G(v)  \frac{\nabla G(u)}{\|\nabla G(u)\|_H^2},\quad v\in H\backslash\{0\}. \] 
By the assumptions (A1)-(A2), $\psi$ is a $C^1$ mapping from $H\backslash\{0\}$ to $H$. Clearly, $\langle G'(u),\psi(v)\rangle=G(v)$, and therefore
\begin{align}\label{eq:psi-v-N}
\psi(v)\in X_u\quad \Longleftrightarrow \quad v\in \mathcal{N}. 
\end{align}
Moreover, $\psi(u) = 0$ and $\psi'(u) = \mathrm{Id}$, then $\psi$ is locally invertible at  $u$ and is a $C^1$-diffeomorphism between a neighborhood $U \subset H\backslash\{0\}$ of $u$ (which can be taken to be far from 0 by the assumption (A3)) and a neighborhood $V \subset H$ of $0$. Note that the orthogonal direct sum decomposition $H=X_u\oplus\text{span}\{\nabla G(u)\}$. One can take $\tilde{V} = V_1 \times V_2\subset V$ for some open neighborhoods $V_1 \subset  X_u$ and $V_2 \subset \text{span}\{\nabla G(u)\}$ of $0$. Then $\tilde{U}:=\psi^{-1}(\tilde{V})\subset U$ is also open in $H$, and $(\tilde{U},\psi)$ gives a chart at $u$ for $H$. By \eqref{eq:psi-v-N}, $\psi$ satisfies $\psi(\tilde{U} \cap \mathcal{N})= \psi(\tilde{U})\cap X_u = V_1$. Then the restriction of $\psi$ on $\mathcal{N}$ induces a bijection $\psi_1:\tilde{U} \cap \mathcal{N}\to V_1$. Since $\psi:\tilde{U}\to \tilde{V}$ is a $C^1$-diffeomorphism, the collection of pairs $(\tilde{U} \cap \mathcal{N},\psi_1)$ obtained in the above manner forms an atlas for $\mathcal{N}$, of class $C^1$ \cite[II, Lemma~2.1]{1995differential}. According to the definition, the closed subset $\mathcal{N}$ of $H$ is a closed $C^1$ submanifold of $H$. 
\end{proof}

The {\em tangent space} of the Nehari manifold $\mathcal{N}$ at $u\in\mathcal{N}$ can be expressed as
\begin{align}\label{eq:TuN}
 T_u\mathcal{N} %= \text{ker}\left(G'(u)\right) 
 = \{\xi \in H : \langle G'(u),\xi\rangle = 0\}
  =\{\xi \in H: \left(\nabla G(u), \xi\right)_H = 0\},
\end{align}
which is a closed, linear subspace of $H$ of codimension one (hence a Hilbert space). For each tangent vector $\xi \in T_u \mathcal{N}$, there is at least one $C^1$ curve $\gamma:[a,b]\to\mathcal{N}$ with $a<0<b$, s.t. $\gamma(0)=u$ and $\gamma'(0)=\xi$. Such a curve $\gamma$ is said to realize the tangent vector $\xi$. As a submanifold of $H$, the Nehari manifold $\mathcal{N}$ naturally admits a Riemannian structure with $(\cdot,\cdot)_H$ as the {\em Riemannian metric}. More precisely, $T_u\mathcal{N}$ is naturally equipped with the restriction of $(\cdot,\cdot)_H$ on $T_u\mathcal{N}\times T_u\mathcal{N}$ as its inner product. Based on these, the Riemannian gradient of $E$ on the Nehari manifold $\mathcal{N}$ is defined as follows \cite{0Nonlinear}, which is the steepest ascent direction in the tangent space and a central concept for considering critical points of $E$ on the Nehari manifold $\mathcal{N}$.

\begin{definition}[Riemannian gradient]\label{def:Riemannian-gradient}
Under the assumptions (A1)-(A3), for $u\in\mathcal{N}$, the {\em Riemannian gradient} of $E$ at $u$, denoted by $\nabla_{\mathcal{N}}E(u)$, is the unique element in the tangent space $T_u \mathcal{N}$ s.t.
\begin{align}\label{eq:Rgrad-def}
(\nabla_{\mathcal{N}}E(u) , \xi)_{H}=\langle E'(u),\xi \rangle, \quad \forall\; \xi \in T_{u} \mathcal{N}.
\end{align}
\end{definition}

From the definition, the Riemannian gradient $\nabla_{\mathcal{N}}E(u)$ is exactly the $H$-orthogonal projection of the $H$-gradient $\nabla E(u)\in H$ onto the tangent space $T_{u} \mathcal{N}$. Let $P_u:H\to T_u\mathcal{N}$ be the $H$-orthogonal projection at $u\in\mathcal{N}$. Direct computations yield that
\begin{align}\label{Projection}
P_u\phi = \phi - \frac{(\nabla G(u),\phi)_H}{\|\nabla G(u)\|_H^2}\nabla G(u), \quad \forall\; \phi \in H.
\end{align}
Therefore, the Riemannian gradient $\nabla_{\mathcal{N}}E(u)$ at $u\in\mathcal{N}$ can be addressed as 
\begin{equation}\label{eq:Rgrad=ProjHgrad}	
\nabla_{\mathcal{N}}E(u)=P_u(\nabla E(u))=\nabla E(u)-\frac{(\nabla G(u),\nabla E(u))_{H}}{\|\nabla G(u)\|_H^2}\nabla G(u),\quad u\in\mathcal{N}. 
\end{equation}

Based on the connection between $\nabla_{\mathcal{N}}E(u)$ and $\nabla E(u)$, we can draw a significant conclusion in the following lemma that the Nehari manifold $\mathcal{N}$ is a {\em natural constraint} in the sense that all nontrivial critical points of $E$ in $H$ are the same as critical points on $\mathcal{N}$. Hereafter, a critical point on $\mathcal{N}$ is referred to a point on $\mathcal{N}$ at which the Riemannian gradient vanishes.
	\begin{lemma}\label{lem:naturalcons}
		Under the assumptions (A1)-(A3), the followings are equivalent to each other:
      (i) $u\in H\backslash\{0\}$ with $\nabla E(u)=0$; 
      (ii) $u\in \mathcal{N}$ with $\nabla_{\mathcal{N}} E(u) = 0$.
	\end{lemma}
	\begin{proof}
The proof ``(i)~$\Rightarrow$~(ii)" follows immediately from the definition of $\mathcal{N}$ and the expression of $\nabla_{\mathcal{N}}E(u)$ in \eqref{eq:Rgrad=ProjHgrad}. To prove ``(ii)~$\Rightarrow$~(i)", suppose $u\in\mathcal{N}$ with $\nabla_{\mathcal{N}}E(u) =0$. It implies that $\nabla E(u) = \alpha\nabla G(u)$ for some $\alpha \in \mathbb{R}$. Taking $H$ inner products with $u$ arrives at $(\nabla E(u),u)_{H} = \alpha(\nabla G(u),u)_{H}$. Together with the facts that $(\nabla G(u),u)_{H} = \langle E''(u)u,u\rangle < 0$ and $(\nabla E(u),u)_{H}=\langle E'(u),u\rangle=0$ for all  $u\in\mathcal{N}$, we have $\alpha = 0$ and $\nabla E(u) = 0$. Obviously, $u\in\mathcal{N}\subset H\backslash\{0\}$. (i) is obtained.
\end{proof}

Furthermore, we can establish connections between local minimizers on the Nehari manifold $\mathcal{N}$ and 1-saddles in $H$. This actually provides a variational characterization for 1-saddles of  $E$ in $H$. 

\begin{theorem}\label{theorem: 2.1}
Under the assumptions (A1)-(A3), all local minimizers of $E$ on the Nehari manifold $\mathcal{N}$ are 1-saddles of $E$ in $H$, and all nontrivial and nondegenerate 1-saddles of $E$ in $H$ are strict local minimizers of $E$ on the Nehari manifold $\mathcal{N}$.
\end{theorem} 

\begin{proof}
Let $u_*\in \mathcal{N}$ be a local minimizer of $E$ on $\mathcal{N}$. Then $\nabla_{\mathcal{N}}E(u_*) = 0$ and $E'(u_*) = 0$ by Lemma~\ref{lem:naturalcons}. For every tangent vector $\xi \in T_{u_*}\mathcal{N}$ at $u_*$ with $\gamma(s)$ a $C^1$ curve realizing it, the Taylor expansion of $E$ at $u_*$ results in
\begin{align}\label{eq:taylor_expan_of_E}
E(\gamma(s)) 
&= E(u_*) 
+\frac{1}{2}\left\langle E''(u_*)(\gamma(s)-u_*),\gamma(s)-u_*\right\rangle + o(\|\gamma(s)-u_*\|_H^2) \\
&=E\left(u_*\right)+\frac{s^2}{2}\left\langle E^{\prime \prime}\left(u_*\right) \xi, \xi\right\rangle+o\left(s^2\|\xi\|_H^2\right). \nonumber
\end{align}
Due to $u_*$ minimizing $E$ on $\mathcal{N}$, $E(\gamma(s)) \geq E(u_*)$ when $|s|$ is small enough. Taking the limit $s\to0$ yields $\langle E''(u_*)\xi,\xi\rangle \geq 0$. The arbitrariness of $\xi \in T_{u_*}\mathcal{N}$ implies that the linear operator $E''(u_*)$ is positive semi-definite on the tangent subspace $T_{u_*}\mathcal{N}$. This means $\mathrm{MI}(u_*)\leq 1$ since $T_{u_*}\mathcal{N}$ is a subspace of $H$ of codimension one. Further, the assumption $\langle E''(u_*)u_*,u_* \rangle <0$ shows $\mathrm{MI}(u_*) \geq 1$. In conclusion, $\mathrm{MI}(u_*) =1$. 
The first part is proved.

The rest is to verify that each nontrivial and nondegenerate 1-saddle of $E$ in $H$, say $u_0\in H\backslash\{0\}$, must be a strict local minimizer on the Nehari manifold $\mathcal{N}$. From $\mathrm{MI}(u_0) =1$, $\langle E''(u_0)u_0,u_0 \rangle<0$ and the nondegeneracy of $u_0$, we conclude that
\begin{align} \label{eq:MI-saddle}
\langle E''(u_0)\xi,\xi\rangle>0, \quad \forall\; \xi\in T_{u_0}\mathcal{N}\backslash\{0\}. 
\end{align}
Moreover, note that $\mathcal{N}$ is locally $C^1$-diffeomorphic to the tangent space $T_{u_0}\mathcal{N}$ and admits a local parametrization in a neighborhood $U\subset\mathcal{N}$ of $u_0$ given by $(V,\varphi)$, where $V\subset T_{u_0}\mathcal{N}$ is a neighborhood of $0\in T_{u_0}\mathcal{N}$ and $\varphi:V\to U$ is a $C^1$-diffeomorphism with $\varphi(0)=u_0$ and $\varphi'(0)=\mathrm{Id}$. For each $w\in U\backslash\{u_0\}$, there is a unique $\xi\in V\backslash\{0\}$ s.t. $w=\varphi(\xi)=\varphi(0)+\varphi'(0)\xi+o(\|\xi\|_H)=u_0+\xi+o(\|\xi\|_H)$. Applying the Taylor expansion of $E$ at $u_0$, one obtains
\begin{align*}
E(w)
&= E(u_0)+\frac{1}{2}\left\langle E''(u_0)(w-u_0),w-u_0\right\rangle + o\left(\|w-u_0\|_H^2\right) \\
&= E(u_0)+\frac{1}{2}\left\langle E''(u_0)\xi,\xi\right\rangle + o\left(\|\xi\|_H^2\right).
\end{align*} 
Paying attention to $\left\langle E''(u_0)\xi,\xi\right\rangle>0$ and $\|\xi\|_H=O(\|w-u_0\|_H)$, one can find a sufficiently small neighborhood $U_1\subset U\subset\mathcal{N}$ of $u_0$ s.t. $E(w)>E(u_0)$ for all $w\in U_1\backslash\{u_0\}$. Consequently, $u_0$ is a strict local minimizer of $E$ on $\mathcal{N}$.
\end{proof}

\begin{remark}
We mention a classical result for the variational energy functional associated to a special kind of semilinear elliptic equations established in 1989 in \cite{Ni1989}, where the global minimizer on the Nehari manifold was proven to be a 1-saddle. Compared with that result, Theorem~\ref{theorem: 2.1} is concerned with local minimizers on $\mathcal{N}$ and a general functional $E$.
\end{remark}

Following the Palais-Smale (PS) condition \cite{PS1964BAMS,Rabinowitz1986} in the critical point theory, we introduce the following concept for the {\em PS condition on the Nehari manifold}, which is vital for the existence of minimizers on the Nehari manifold. It will be also useful for the convergence analysis of the proposed algorithm framework in Section~\ref{sec:Convergence}.

\begin{definition}[PS condition on the Nehari manifold]\label{PScondition_M}
Under the assumptions (A1)-(A3), $E$ is said to satisfy the PS condition on the Nehari manifold $\mathcal{N}$ if every sequence $\left\{u_{n}\right\} \subset \mathcal{N}$ s.t. $\left\{E\left(u_{n}\right)\right\}$ is bounded and $\|\nabla_{\mathcal{N}}E\left(u_{n}\right)\|_H \rightarrow 0$ has a convergent subsequence in the $H$-norm.
\end{definition}

The existence of the ground-state critical point (i.e., the global minimizer of $E$ on $\mathcal{N}$) is given in the following theorem. By Theorem~\ref{theorem: 2.1}, it is a sufficient condition for the existence of 1-saddles in $H$. The proof can be completed by the Ekeland's variational principle \cite{1974On}. % and is omitted for brevity. 

\begin{theorem}\label{thm:existence}
Under the assumptions (A1)-(A3), if $E$ satisfies the PS condition on the Nehari manifold $\mathcal{N}$, then there exists a ground-state critical point $u^* \in \mathcal{N}$, i.e., $E(u^*) = \inf_{u\in \mathcal{N}}E(u) $ and $E'(u^*) =0$. 
\end{theorem}
\begin{proof}
From Lemma~\ref{lem:Nehari-submanifold}, the Nehari manifold $\mathcal{N}$ is closed in $H$. Then, $\mathcal{N}$ is a complete metrical space w.r.t.~the distance induced by $\|\cdot\|_H$. In addition, $E$ is bounded from below on $\mathcal{N}$ by the assumptions (A1)-(A2). According to the Ekeland's variational principle \cite{1974On}, for any $\varepsilon>0$, there exist some points $u_{\varepsilon}\in\mathcal{N}$ s.t.
\begin{align}
\label{eq:Ekeland-Eu-eps2}
& E(u_{\varepsilon})\leq \inf_{u\in\mathcal{N}}E(u)+\varepsilon^2, 
 \end{align}
 and
 \begin{align}
\label{eq:Ekeland-EwEu}
& E(u)\geq E(u_{\varepsilon})-\varepsilon\|u-u_{\varepsilon}\|_H,\quad\forall \;u\in\mathcal{N}.
\end{align}
Recall that $\mathcal{N}$ is a $C^1$ submanifold of $H$ by Lemma~\ref{lem:Nehari-submanifold}. It admits a local parametrization in a neighborhood $U\subset\mathcal{N}$ of $u_{\varepsilon}$ given by $(V,\varphi)$, where $V\subset T_{u_{\varepsilon}}\mathcal{N}$ is a neighborhood of $0\in T_{u_{\varepsilon}}\mathcal{N}$ and $\varphi:V\to U$ is a $C^1$-diffeomorphism with $\varphi(0)={u_{\varepsilon}}$ and $\varphi'(0)=\mathrm{Id}$. Hence, for each $u\in U$, $\xi = \varphi^{-1}(u) \in V$ satisfies
\[ u=\varphi(\xi)=\varphi(0)+\varphi'(0)\xi+o(\|\xi\|_H)=u_{\varepsilon}+\xi+o(\|\xi\|_H). \]
Noticing the definition of the Riemannian gradient $\nabla_{\mathcal{N}}E(u_{\varepsilon})$, the Taylor expansion of $E$ at $u_{\varepsilon}$ states that
\begin{align*}
E(u)-E(u_{\varepsilon})
=\langle E'(u_{\varepsilon}),\xi\rangle+o(\|\xi\|_H)=(\nabla_{\mathcal{N}}E(u_{\varepsilon}),\xi)_H+o(\|\xi\|_H).
\end{align*}
Utilizing \eqref{eq:Ekeland-EwEu} and the local parametrization $(V,\varphi)$ in the neighborhood $U\subset\mathcal{N}$ of $u_{\varepsilon}$, one can derive
\begin{align}\label{eq:Ekeland-EwEu-xi}
(\nabla_{\mathcal{N}}E(u_{\varepsilon}),\xi)_H\geq -\varepsilon\|\xi\|_H+o(\|\xi\|_H),\quad \forall \;\xi\in V.
\end{align}
Taking $\xi=-s\nabla_{\mathcal{N}}E(u_{\varepsilon})\in V$ in \eqref{eq:Ekeland-EwEu-xi} and letting $s\to0^+$, it yields
\begin{align}\label{eq:Ekeland-DEu-eps}
\|\nabla_{\mathcal{N}}E(u_{\varepsilon})\|_H\leq\varepsilon.
\end{align}
Hence, in view of \eqref{eq:Ekeland-Eu-eps2} and \eqref{eq:Ekeland-DEu-eps}, taking $\varepsilon=1/n$, $n\in\mathbb{N}^+$, one can get a minimizing sequence $\{u_n\}\subset\mathcal{N}$ satisfying $E(u_n)\to\inf_{u\in \mathcal{N}}E(u)$ and $\nabla_{\mathcal{N}}E(u_n)\to0$ as $n\to\infty$. Due to the assumption that $E$ satisfies the PS condition on the Nehari manifold $\mathcal{N}$, the sequence $\{u_n\}$ contains a convergent subsequence. As a result, the existence of a point $u^*\in\mathcal{N}$ s.t. $E(u^*) = \inf_{u\in \mathcal{N}}E(u)$ and $\nabla_{\mathcal{N}}E(u^*)=0$ is guaranteed immediately. Furthermore, from Lemma~\ref{lem:naturalcons}, one has $E'(u^*) =0$. 
\end{proof}

Theorems \ref{theorem: 2.1} and \ref{thm:existence} establish the mathematical justification for finding 1-saddles via the minimization of $E$ on the associated Nehari manifold $\mathcal{N}$. As a result, the  1-saddles can be numerically grasped by suitable minimization algorithms on the Nehari manifold $\mathcal{N}$.

\section{Nehari manifold optimization method (NMOM)}\label{section 3}
In this section, we introduce the NMOM, a framework of Riemannian optimization to minimize $E$ on the Nehari manifold $\mathcal{N}$ for finding 1-saddles of $E$ in $H$. 

\subsection{Algorithm framework of the NMOM} 
We begin with some definitions for the iterative scheme of the NMOM. 
\begin{definition}[Descent direction]
Under the assumptions (A1)-(A3), a {\em descent direction} at $u\in \mathcal{N}$ is a tangent vector $\xi\in T_u\mathcal{N}$ satisfying $(\nabla_{\mathcal{N}}E(u), \xi)_H <0$. 
\end{definition}

Since the Nehari manifold $\mathcal{N}$ lacks linear algebraic structures, the standard linear iterative scheme of the line-search methods in vector spaces, i.e., $ u_{n+1} = u_n +\alpha_n d_n$, is infeasible, where $\alpha_n >0$ is the step-size and $d_n $ is the search direction. Therefore, the {\em retraction}, a vital concept for the construction of the iterative scheme in our algorithm framework, is introduced. Actually, it is an extension of classical retractions in the finite-dimensional Riemannian optimization \cite{2008Optimization,boumal2023intromanifolds} to guarantee each iterative point staying on the Nehari manifold $\mathcal{N}$.Denote the {\em tangent bundle} of the Nehari manifold $\mathcal{N}$ as $T\mathcal{N} := \bigcup_{u\in \mathcal{N}}\{(u,\xi)\left|\right.$ $ \xi \in T_u\mathcal{N}\}$, the retraction is defined as follows.

\begin{definition}[Retraction]\label{retraction definition}
Under the assumptions (A1)-(A3), a {\em retraction} on the Nehari manifold $\mathcal{N}$ is a $C^1$ mapping from $T\mathcal{N}$ to $\mathcal{N}$ s.t. for all $u\in\mathcal{N}$, the restriction of $\mathcal{R}$ on the tangent space ${T_u\mathcal{N}}$, denoted as $\mathcal{R}_{u}=\mathcal{R}(u,\cdot)$, satisfies 
\begin{enumerate}[(i)]
\item $\mathcal{R}_{u}\left(0\right)=u$;
\item $\left.\frac{d}{ds}\mathcal{R}_u(s\xi)\right|_{s=0}=\xi$, \; $\forall\;\xi\in T_u\mathcal{N}$. 
\end{enumerate}
\end{definition}

\begin{figure}[!ht]
\centering
\vspace{-3ex}
\includegraphics[width=0.6\textwidth,height=0.4\textwidth]{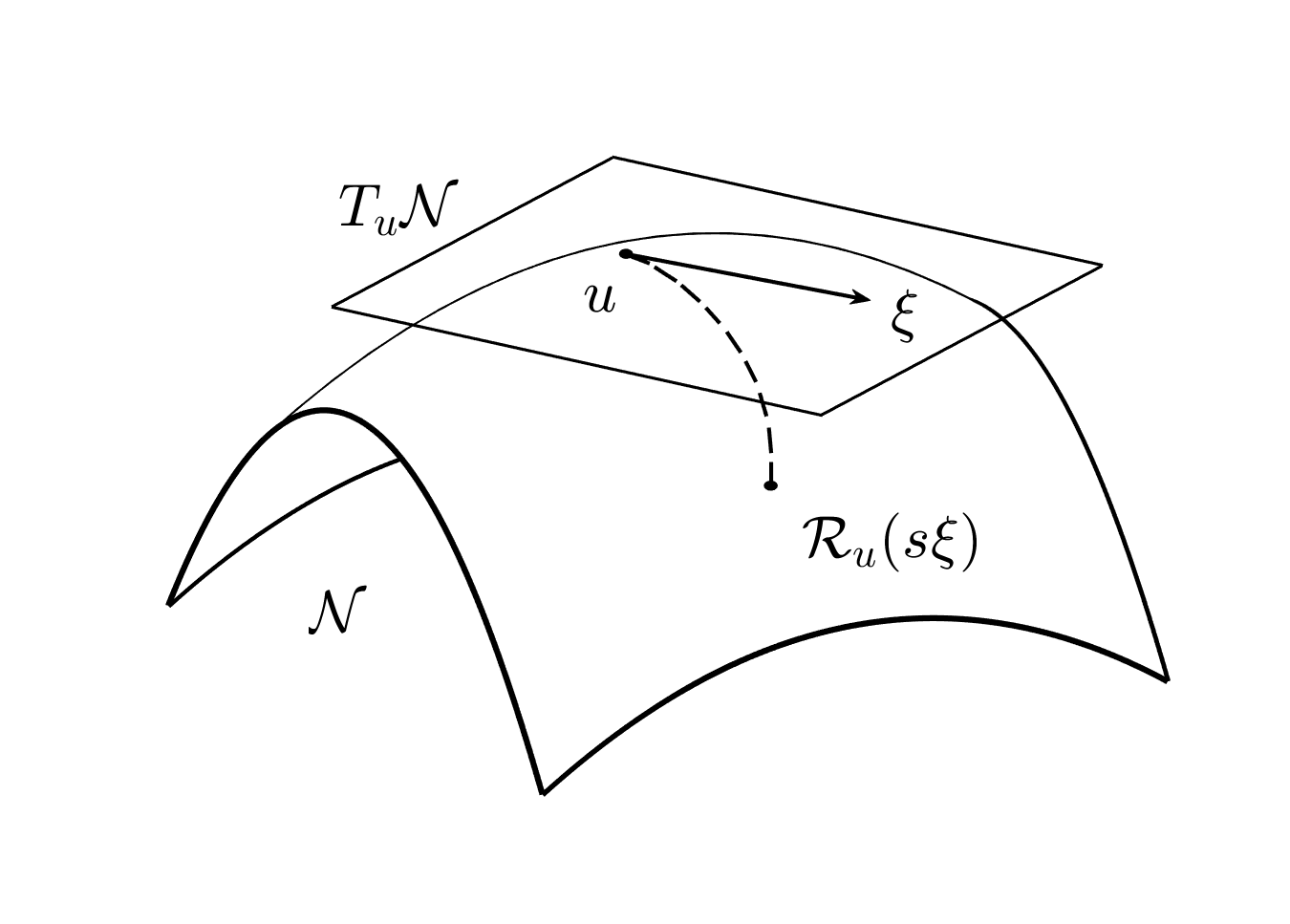}
\vspace{-5ex}
\caption{Illustration of a retraction $\mathcal{R}$ on the Nehari manifold.}\label{fig.Retraction_2}
\end{figure}

\begin{remark}
% \begin{enumerate}[(1)]
\setlength{\itemsep}{1pt}
\setlength{\parskip}{0pt}
\setlength{\parsep}{0pt}
By Definition~\ref{retraction definition}, every tangent vector $\xi\in T_u\mathcal{N}$ is realized by the curve $\gamma(s):=\mathcal{R}_u(s\xi)$. It allows one moving away from $u$ while always staying on the Nehari manifold $\mathcal{N}$ (see Figure~\ref{fig.Retraction_2}). In addition, conditions (i) and (ii) in Definition~\ref{retraction definition} ensure that $\mathcal{R}_u(s\xi)$ is close to $u$ with a linear approximation as $\mathcal{R}_u(s\xi)\approx u+s\xi$ when $\|s\xi\|_H$ is sufficiently small. These guarantee the stability and local linearity of the iteration based on a retraction. Moreover, the $C^1$ property of $\mathcal{R}$ means that there exists an $C^1$ extension $\hat{\mathcal{R}}$ defined on an open subset of $H \times H$ containing $T\mathcal{N}$, s.t.  $\hat{\mathcal{R}}\big|_{T\mathcal{N}} = \mathcal{R}$.
\end{remark}

Utilizing a retraction $\mathcal{R}$ on the Nehari manifold $\mathcal{N}$, the iterative scheme of the NOMM is expressed as
\begin{align}\label{eq:NMOM-IterativeScheme}
u_{n+1} = \mathcal{R}_{u_n}(\alpha_n\xi_n),\quad n\geq0, 
\end{align}
with $u_0\in\mathcal{N}$ an initial guess, $\alpha_n>0$ a step-size and $\xi_n \in T_{u_n}\mathcal{N}$ a 
descent direction at $u_n$. Then, we propose the algorithm framework for the NOMM in Algorithm~\ref{alg:NMOM}. 

\begin{algorithm}[!ht]
	\caption{Algorithm framework of the NMOM}\label{alg:NMOM}
 \begin{enumerate}[\bf Step~1.]
 \vspace{-\topsep}
\setlength{\itemsep}{1pt}
\setlength{\parskip}{0pt}
\setlength{\parsep}{0pt}
	\item Given the initial guess $u_0 \in \mathcal{N}$ and a precision tolerance $\varepsilon_{\mathrm{tol}}>0$, set $n=0$.
	\item If $\|\nabla_{\mathcal{N}}E(u_n)\|_H < \varepsilon_{\mathrm{tol}}$ (or other criterion is satisfied), then stop; otherwise, go to {\bf Step 3}.
	\item Compute a descent direction $\xi_n\in T_{u_n}\mathcal{N}$ at $u_n\in \mathcal{N}$.
	\item Find a step-size $\alpha_n>0$ by a certain step-size search strategy.
	\item Do the iterative scheme \eqref{eq:NMOM-IterativeScheme} with a certain retraction $\mathcal{R}$ to get $u_{n+1}$. Update $n:=n+1$, and then go to {\bf Step 2}.
\end{enumerate}
\vspace{-\topsep}
\end{algorithm}

We remark here that various suitable retractions, search directions and step-size search rules are allowed to implement Algorithm~\ref{alg:NMOM}. In the rest of this section, we discuss the practical retraction and step-size search rules.

\subsection{Nehari retraction}\label{sec:Nehari retraction}
In this subsection, we propose a practical retraction, called the {\em Nehari retraction}, by utilizing the geometric structure of the Nehari manifold $\mathcal{N}$. To illustrate this in details, we first present the following fact that there exists a $C^1$ mapping pulling all nonzero points in $H$ to the Nehari manifold $\mathcal{N}$.

\begin{lemma}\label{lem:sigma}
	Under the assumptions (A1)-(A3), for every $v\in H\backslash\{0\}$, there exists a unique $\rho(v)>0$ s.t. $\rho(v)v \in \mathcal{N}$. Furthermore, $\rho:v \mapsto \rho(v)$ is a $C^1$ mapping from $H\backslash\{0\}$ to $\mathbb{R}^+$.
\end{lemma}
\begin{proof}
The existence and uniqueness of $\rho(v)>0$ straightforwardly follow from (A1) by letting $\rho(v):=s_v$. Consider a $C^1$ mapping $F:\mathbb{R}^+\times(H\backslash\{0\})\to\mathbb{R}$ defined as $F(s,v)=G(sv)$. Under the assumptions (A1)-(A3), for every $v\in H\backslash\{0\}$, one has 
\[
F(s_v,v)=G(s_vv)=0,\quad F'_s(s_v,v)= \langle G'(s_vv),v \rangle = \frac{1}{s_v}\langle E''(s_vv)s_vv,s_vv\rangle
<0. \] 
The implicit function theorem concludes immediately that $\rho(v) := s_v$ is of $C^1$ on a neighborhood of $v$. The proof is completed with the arbitrariness of $v\in H\backslash\{0\}$.
\end{proof}

Inspired by Lemma~\ref{lem:sigma}, a practical retraction \eqref{Nehari retraction} is introduced in Lemma~\ref{lemma:rho}, called the Nehari retraction.

\begin{lemma}\label{lemma:rho}
Under the assumptions (A1)-(A3), the mapping
\begin{align}\label{Nehari retraction}
\mathcal{R}(u,v) :=\rho(u+v)(u+v),\quad \forall\;(u,v)\in H \times H, \, u+v \neq 0
\end{align}
with $\rho$ defined in Lemma~\ref{lem:sigma}, is a retraction on the Nehari manifold $\mathcal{N}$. 
\end{lemma}
\begin{proof}
It is noted that $u+\xi \neq 0$, $\forall \;(u,\xi) \in T\mathcal{N}$, due to the fact
\[ \langle G'(u),u+\xi \rangle = \langle G'(u),u\rangle +\langle G'(u),\xi \rangle = \langle G'(u), u\rangle <0,\quad \forall\; (u,\xi) \in T\mathcal{N}. \]
Therefore the mapping $\mathcal{R}$~\eqref{Nehari retraction} is well defined on $T\mathcal{N}$. According to Lemma~\ref{lem:sigma}, it is easy to see that the mapping $\mathcal{R}(u,v)$ given in \eqref{Nehari retraction} is of $C^1$ and $\mathcal{R}(T\mathcal{N})\subset \mathcal{N}$.  Clearly, for all $u \in \mathcal{N}$,  $\rho(u)=1$ and $\mathcal{R}_u(0)=\rho(u)u=u$. It remains to verify $\left.\frac{d}{ds}\mathcal{R}_u(s\xi)\right|_{s=0}=\xi$, $\forall\;\xi\in T_u\mathcal{N}$. In fact, the definition of $\rho$ 
states that
\[ G(\rho(u+s\xi)(u+s\xi))=0, \quad \forall \;(u,\xi) \in T\mathcal{N},\;  \forall s\in \mathbb{R}. \]
Differentiating it w.r.t.~$s$ and noticing $\rho(u)=1$, it holds for all $(u,\xi) \in T\mathcal{N}$ that
\begin{align*}
0 = \frac{d}{ds}G(\rho(u+s\xi)(u+s\xi))\Big|_{s=0}
&= \left\langle G'(\rho(u)u), \frac{d}{ds}\left(\rho(u+s\xi)(u+s\xi)\right)\Big|_{s=0} \right\rangle \\
&= \big\langle G'(u),u\big\rangle \frac{d}{ds}\rho(u+s\xi)\Big|_{s=0} + \big\langle G'(u),\xi\big\rangle.
\end{align*}
It follows from $\langle G'(u),u \rangle <0$ and $\langle G'(u),\xi \rangle =0$ that $\frac{d}{ds}\rho(u+s\xi)\big|_{s=0} = 0$. Therefore,
\[ \frac{d}{d s} \mathcal{R}_u(s \xi)\Big|_{s=0}=\left(\frac{d}{d s} \rho(u+s \xi)\Big|_{s=0}\right) u+\rho(u) \xi=\xi ,\quad \forall\; \xi \in T_u\mathcal{N}, \]
which completes the proof.
\end{proof}

\begin{figure}[!ht]
\centering
\vspace{-2ex}
{\includegraphics[width=0.42\textwidth,height=0.28\textwidth]{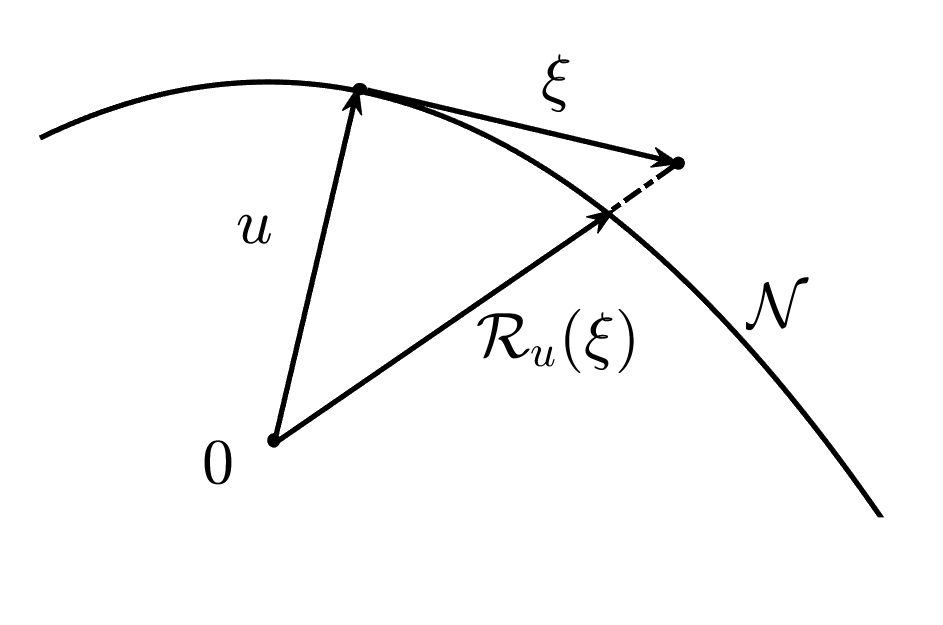}}
\vspace{-4ex}
\caption{Geometric landscape of the Nehari retraction $\mathcal{R}$ \eqref{Nehari retraction} in the subspace $\mathrm{span}\{u,\xi\}$.}\label{fig.out_retraction}
\end{figure}

As shown in Figure~\ref{fig.out_retraction}, the computation of the Nehari retraction $\mathcal{R}(u,\xi)$~\eqref{Nehari retraction} is done by gaining $\rho(v)$ for  $v = u +\xi$, which is equivalent to find the unique positive root of the equation $\langle E'(sv),v\rangle = 0$ w.r.t.~$s>0$. Generally, a nonlinear algebraic equation has to be numerically solved in each computation of the Nehari retraction. While for all examples in our numerical experiments in Section~\ref{Numerical experiments}, $\rho(v)$ is formulated explicitly as a simple expression and thus the computation of the Nehari retraction \eqref{Nehari retraction} is very efficient.

\subsection{Nonmonotone step-size search rules}\label{sec:Nonmonotone rules}
In terms of the convergence and efficiency of the algorithm, the selection of the step-size is crucial. Inspired by the nonmonotone step-size search rule in the Euclidean optimization \cite{Hongchao2004A} and its variant in the finite-dimensional manifold optimization \cite{2017AdaptiveBB}, we consider the following nonmonotone step-size search rule for Algorithm~\ref{alg:NMOM}: Given a constant $\sigma \in (0,1)$, find a step-size $\alpha_n>0$ s.t.
	\begin{align}
	\label{nonmonotone armijo step-size}
	E(\mathcal{R}_{u_n}(\alpha_n\xi_n)) &\leq C_n + \sigma  \alpha_n \left(\nabla_{\mathcal{N}} E(u_n),\xi_n\right)_{H},\quad n\geq0,
	\end{align}
where $C_{n}$, with a given parameter $\varrho \in [0,1)$, is computed iteratively by
\begin{align}\label{eq:CkQk-update}
 \begin{cases}
 C_0=E(u_0), \; Q_0 = 1,\\
 Q_{j+1} = \varrho Q_j +1, \; 
 C_{j+1} = \dfrac{\varrho Q_jC_j+E(u_{j+1})}{Q_{j+1}},\quad j=0,1,\ldots,n-1. \\
 \end{cases}
\end{align}
Clearly, $Q_n=1+\sum_{j=1}^n\varrho^j$ and $C_{n} = \frac{1}{Q_n}\big({E(u_n)+\sum_{j=1}^n\varrho^{j}E(u_{n-j})}\big)$. Thus, $C_n$ is a strict convex combination of $E(u_0),E(u_1),\cdots, E(u_n)$ when $\varrho\in (0,1)$. In the special case of $\varrho =0$, $C_n=E(u_n)$ and \eqref{nonmonotone armijo step-size} reduces to the Armijo-type search condition, which is monotone:
\[	E(\mathcal{R}_{u_n}(\alpha_n\xi_n)) \leq E(u_n) + \sigma  \alpha_n \left(\nabla_{\mathcal{N}} E(u_n),\xi_n\right)_{H},\quad n\geq0. \]

The following lemma concludes that the step-size satisfying \eqref{nonmonotone armijo step-size} always exists. The proof is based on the mathematical induction and borrows ideas in \cite{Hongchao2004A} and \cite{LXY2024JCM}.

\begin{lemma}[Existence of step-size]\label{lem:steo-size exists}
Under the assumptions (A1)-(A3), let $u_0 \in \mathcal{N}$, $C_0 = E(u_0)$ and $\mathcal{R}$ be a retraction to $\mathcal{N}$.  If $\nabla_{\mathcal{N}} E(u_n) \neq 0$ and $\xi_n \in T_{u_n}\mathcal{N}$ is descent direction, $n=0,1,\cdots$, then there always exists $\alpha_n^A$, s.t. all $\alpha_n \in (0,\alpha_n^A)$ satisfy the nonmonotone search rule \eqref{nonmonotone armijo step-size}.
\end{lemma}

\begin{proof}
For $(u,\xi)\in T\mathcal{N}$, defining $\Phi(\alpha) = E(\mathcal{R}_u(\alpha\xi))$, $\alpha\in\mathbb{R}$, one has $\Phi(0)=E(u)$, $\Phi'(0)=\langle E'(u),\xi\rangle=\left(\nabla_{\mathcal{N}} E(u),\xi\right)_{H}$, and $E(\mathcal{R}_u(\alpha\xi)) = E(u)+\alpha\left(\nabla_{\mathcal{N}} E(u),\xi\right)_{H} + o(\alpha)$. Hence, if $\xi$ is a descent direction at $u$, there exists $\alpha^A>0$ s.t.
\begin{align}\label{armijio exists}
E(\mathcal{R}_{u}(\alpha\xi)) < E(u) + \sigma\alpha \left(\nabla_{\mathcal{N}} E(u),\xi\right)_{H}, \quad\forall \;\alpha \in (0,\alpha^A).
\end{align}
The proof continues by mathematical induction on $n$. Since $C_0 = E(u_0)$, the conclusion is trivial for $n=0$ by setting $(u,\xi)=(u_0,\xi_0)$ in \eqref{armijio exists}. For $n\geq1$, supposing that $\alpha_j>0$ ($j=0,1,\cdots,n-1$) satisfy \eqref{nonmonotone armijo step-size} with $u_{j+1} = \mathcal{R}_{u_{j}}(\alpha_j\xi_j)$, $\xi_j \in T_{u_j} \mathcal{N}$ and $C_j$ defined in \eqref{eq:CkQk-update}, we proceed to prove that there exists $\alpha_n^A>0$ s.t. \eqref{nonmonotone armijo step-size} holds for all $\alpha_n\in (0,\alpha_n^A)$. Noting that $\xi_{n-1}\in T_{u_{n-1}}\mathcal{N}$ is a descent direction at $u_{n-1}$, \eqref{eq:CkQk-update} yields $E(u_n)\leq C_{n-1}$ and
\[ C_n = \frac{\varrho Q_{n-1}C_{n-1} +E(u_{n})}{Q_{n}} \geq \frac{\varrho Q_{n-1}E(u_{n})+ E(u_n)}{Q_{n}} = E(u_n). \]
Setting $(u,\xi)=(u_n,\xi_n)$ in \eqref{armijio exists}, the proof is completed immediately.
\end{proof}

In addition, under the nonmonotone step-size search rule \eqref{nonmonotone armijo step-size}, a basic property holds as below.
\begin{lemma}\label{lem:EukCkEu0}
Under the assumptions (A1)-(A3), letting $\{u_n\}$ be the sequence generated by Algorithm~\ref{alg:NMOM} with the nonmonotone step-size search rule \eqref{nonmonotone armijo step-size}, 
the following inequalities hold,
\begin{align}
E(u_n)\leq C_n\leq C_{n-1}\leq E(u_0),\quad n\geq 1.
\end{align}
\end{lemma}
\begin{proof}
The step-size search rule \eqref{nonmonotone armijo step-size} implies that $E(u_n)\leq C_{n-1}$ for all $n\geq1$. In addition, \eqref{eq:CkQk-update} means that $C_n$ is a convex combination of $E(u_n)$ and $C_{n-1}$. Thus, $E(u_n)\leq C_n\leq C_{n-1}$ for all $n\geq1$. Consequently, the sequence $\{C_n\}$ is monotonically nonincreasing. The remaining inequalities $C_n\leq E(u_0)$ for all $n\geq0$ are guaranteed with the initialization $C_0=E(u_0)$.
\end{proof}

In practice, a standard backtracking strategy can be applied in step~4 of Algorithm~\ref{alg:NMOM} to find a suitable step-size $\alpha_n>0$ satisfying the nonmonotone step-size search rule \eqref{nonmonotone armijo step-size}. The detail is summarized in Algorithm~\ref{alg:nonmonotone-backtracking}. 

\begin{algorithm}[!ht]
\caption{Backtracking algorithm for searching a step-size $\alpha_n$ satisfying \eqref{nonmonotone armijo step-size}}\label{alg:nonmonotone-backtracking}
Given constants $0<\alpha_{\min}<\alpha_{\max}$, $\varrho \in[0,1)$ and a backtracking factor $\beta\in(0,1)$.
\begin{enumerate}[\bf{Step~4-}1.]
\vspace{-\topsep}
\setlength{\itemsep}{1pt}
\setlength{\parskip}{0pt}
\setlength{\parsep}{0pt}
\item Compute $Q_n$ and $C_n$ according to \eqref{eq:CkQk-update}.
\item Choose a trial step-size $ \alpha_n^0\in[\alpha_{\min},\alpha_{\max}]$ and set $\alpha_n:= \alpha_n^0$.
\item If $\alpha_n$ satisfies \eqref{nonmonotone armijo step-size}, then stop; otherwise, go to {\bf Step~4-4}.
\item Update $ \alpha_n:=\beta \alpha_n$ and go to {\bf Step~4-3}. 
\end{enumerate}
\vspace{-\topsep}
\end{algorithm}

\section{Global Convergence}\label{sec:Convergence}

In this section, we establish the global convergence results for {\em Algorithm~\ref{alg:NMOM} with step-sizes determined by Algorithm~\ref{alg:nonmonotone-backtracking}}, referred to as {\bf Algorithm~\ref*{alg:NMOM}$'$} in the following for simplicity.

\subsection{Main results}
In order to obtain the convergence for {\bf Algorithm~\ref*{alg:NMOM}$'$}, the following concepts of gradient-related and strongly gradient-related sequences for the search direction are required in our analysis. They ensure the sufficient decreasing property of the search direction around noncritical points and then prevent the sequence of iterative points $\{u_n\}$ from accumulating at any noncritical point. The definition of gradient-related sequences extends a similar concept in the finite-dimensional Riemannian optimization (see, e.g., \cite[Definition~4.2.1]{2008Optimization}). 

\begin{definition}\label{def:gradient-related}
Under the assumptions (A1)-(A3), let $\mathcal{R}$ be a retraction on the Nehari manifold $\mathcal{N}$ and let $\{(u_n,\xi_n)\}\subset T\mathcal{N}$ be a sequence.
    \begin{enumerate}[(i)]
\vspace{-.5\topsep}
\setlength{\itemsep}{1pt}
\setlength{\parskip}{0pt}
\setlength{\parsep}{0pt}
    \item The sequence $\{\xi_n\}$ is {\it gradient-related} if for every subsequence $\{u_n\}_{n\in \mathcal{I}}$ ($\mathcal{I}\subset\mathbb{N}$) converging to a noncritical point, the corresponding subsequence $\{\xi_n \}_{n\in \mathcal{I}}$ is bounded in $H$ 
        % (i.e., $\sup_{n\in \mathcal{I}}\|\xi_n\|_H<\infty$) 
        and $\limsup_{n \in \mathcal{I}}(\nabla_{\mathcal{N}}E(u_n),\xi_n)_{H}<0$. 
    \item The sequence $\{\xi_n\}$ is {\it strongly gradient-related} if there exist two constants $c_1,c_2>0$ s.t., for all $n\geq0$,
		\begin{align}\label{eq:strong-grad-related}
		  \left(\nabla_{\mathcal{N}}E(u_n), \xi_{n}\right)_H \leq-c_1\left\|\nabla_{\mathcal{N}}E(u_n)\right\|_H^{2},\quad  \left\|\xi_{n}\right\|_H \leq c_2\left\|\nabla_{\mathcal{N}}E(u_n)\right\|_H.
		\end{align}
        \end{enumerate}
	\end{definition}
\begin{remark}\label{re:strongg4.1}
A strongly gradient-related sequence is a gradient-related sequence with stronger controls by the Riemannian gradient directions. Clearly, the steepest descent direction $\xi_n^{\mathrm{SD}}=-\nabla_{\mathcal{N}}E(u_n)\in T_{u_n}\mathcal{N}$ satisfies the strongly gradient-related conditions. A wide class of other search directions is also allowed, e.g., a preconditioned steepest descent direction of the from $\xi_n^{\mathrm{PSD}}=-\mathcal{B}_n\nabla_{\mathcal{N}}E(u_n)$, where $\mathcal{B}_n$ is a bounded linear operator on $H$ with properties $\mathcal{B}_n\phi\in T_{u_n}\mathcal{N}$ and $c_1\|\phi\|_H^2\leq(\phi,\mathcal{B}_n\phi)_H\leq c_2\|\phi\|_H^2$, $\forall\;\phi\in T_{u_n}\mathcal{N}$. 
\end{remark}

By Lemma~\ref{lem:EukCkEu0}, the sequence of iterative points $\{u_n\}$ in {\bf Algorithm~\ref*{alg:NMOM}$'$} is always contained in the energy sublevel set 
\begin{align}\label{eq:Energy-sublevelset}
\mathcal{E}_{u_0}:=\left\{u\in \mathcal{N}:\, E(u)\leq E(u_0)\right\}.
\end{align}
Thus, if necessary, we only need to focus on the case where the domain of the retraction $\mathcal{R}$ is limited to be $\{(u,\xi)\in T\mathcal{N}:u\in\mathcal{E}_{u_0}\}$ in our convergence analysis.

The global convergence results for {\bf Algorithm~\ref*{alg:NMOM}$'$} are described as below.

\begin{theorem}\label{theorem-global-convergence}
Under the assumptions (A1)-(A3), let $\mathcal{R}$ be a retraction on the Nehari manifold $\mathcal{N}$ and $\{(u_n,\xi_n)\}\subset T\mathcal{N}$ be the sequence generated by {\bf Algorithm~\ref*{alg:NMOM}$'$}. The followings hold:
\begin{enumerate}[\rm(a)]
	  \item Assume that $\{\xi_n\}$ is a gradient-related sequence. Then every accumulation point of $\left\{u_{n}\right\}$ is a critical point of $E$ on the Nehari manifold $\mathcal{N}$, therefore a nontrivial critical point of $E$ in $H$.
	  \item Assume that $\{\xi_n\}$ is a strongly gradient-related sequence, $E$ satisfies the PS condition on the Nehari manifold $\mathcal{N}$, and the retraction $\mathcal{R}$ satisfies, for some constants $\delta>0$ and $M_{\delta}>0$,
	\begin{align}\label{lipschitz continuous}
    	  \|\mathcal{R}(u, \xi)-u\|_H\leq M_{\delta}\|\xi\|_H,\quad \forall \;(u,\xi)\in T\mathcal{N},\;  u\in\mathcal{E}_{u_0},\;\|\xi\|_H\leq\delta.
	\end{align}
	Then the sequence $\{u_n\}$ possesses an accumulation point, i.e., there exists a subsequence of $\{u_n\}$ converging to a critical point $u_*$. 
 \item In addition to (b), if $u_*$ is isolated, the whole sequence $\{u_n\}$ converges to $u_*$.
	\end{enumerate}
\end{theorem}

Theorem~\ref{theorem-global-convergence} ensures that the sequence generated by {\bf Algorithm~\ref*{alg:NMOM}$'$} must have an accumulation point, and every accumulation point must be a critical point. Furthermore, the whole sequence would converge to a particular accumulation point that is an isolated critical point. Note that the assumption \eqref{lipschitz continuous} naturally holds by the $C^1$ property of the retraction when $H$ is finite-dimensional and $\mathcal{E}_{u_0}$ is bounded. While in general infinite-dimensional cases, the assumption \eqref{lipschitz continuous} is not directly guaranteed due to the lack of compactness. Nevertheless, as we will see later, for typical PDEs considered in Section~\ref{Numerical experiments}, the proposed Nehari retraction $\mathcal{R}(u,\xi)$ in \eqref{Nehari retraction} indeed meets the assumption \eqref{lipschitz continuous}.

The proof of Theorem~\ref{theorem-global-convergence} will be presented in Section~\ref{sec:proof-cvg}, and the following lemma will play a crucial role.
\begin{lemma}\label{for lemma2}
Under the assumptions (A1)-(A3), there exists a number $C_{\infty}$ s.t.
\begin{align}
\lim\limits_{n\rightarrow \infty} E(u_n) = \lim\limits_{n\rightarrow \infty}C_n = C_{\infty}, \qquad
\sum_{n=0}^{\infty}  \alpha_n\left|\left(\nabla_{\mathcal{N}}E(u_n), \xi_n\right)_H\right| < \infty. 
\end{align}
\end{lemma}
\begin{proof}
Firstly, since $E(u_n)\geq E(0)$, %(see Remark~\ref{rmk:A1-A3}), 
Lemma~\ref{lem:EukCkEu0} implies that $\{C_n\}$ is monotonically nonincreasing and bounded from below. Thus, $\{C_n\}$ converges to a finite number $C_{\infty}$ with $C_{\infty}\leq C_n$. Noting that $Q_n=1+\sum_{j=1}^{n}\varrho^{j}\to 1/(1-\varrho)=:Q_{\infty}$ as $n\to\infty$, the definition of $C_n$ in \eqref{eq:CkQk-update} leads to $E(u_{n}) = Q_{n}C_{n}-\varrho Q_{n-1}C_{n-1} \to C_{\infty}$ as $n\to\infty$. Secondly, in view of \eqref{nonmonotone armijo step-size}-\eqref{eq:CkQk-update}, we have
\begin{align}\label{eq:Ck-decay}
C_{n+1}=\frac{\varrho Q_{n} C_{n}+E\left(u_{n+1}\right)}{Q_{n+1}}
\leq C_n-\frac{\sigma \alpha_n\left|\left(\nabla_{\mathcal{N}}E(u_n), \xi_n\right)_H\right|}{Q_{n+1}},\quad n\geq 0.
\end{align}  
Summing \eqref{eq:Ck-decay} w.r.t. $n\geq 0$ and noting that $C_{\infty}\leq C_{n+1}\leq C_n\leq E(u_0)$, one gets
\[
\sum_{n=0}^{\infty} \frac{ \alpha_n \left|\left( \nabla_{\mathcal{N}}E(u_n) ,\xi_n\right)_H\right|}{Q_{n+1}} 
\leq \sum_{n=0}^{\infty}\frac{C_n- C_{n+1}}{\sigma} = \frac{E(u_0)-C_{\infty}}{\sigma}<\infty.
\]
With the fact $0<Q_n\leq Q_{\infty}$ for all $n\geq 0$, the proof is completed. 
\end{proof}

\subsection{Proof of Theorem~\ref{theorem-global-convergence}}\label{sec:proof-cvg}

This subsection is devoted to provide the detailed proof of Theorem~\ref{theorem-global-convergence} in three parts. The technical detail in our proof differs from the case in the finite-dimensional Riemannian optimization, see, e.g., \cite[Theorem~4.3.1]{2008Optimization}. Actually, the sequential compactness principle based on the Bolzano-Weierstrass theorem used in \cite{2008Optimization} is only available in finite-dimensional cases. To deal with general Hilbert submanifolds, a {\em weak convergence technique} (see Lemmas~\ref{lem:strong-weak-cvg} and \ref{lem:strong-weak-cvg-DEhat} below) is employed. With this weak convergence technique, the proof of the assertion (a) in Theorem~\ref{theorem-global-convergence} can be achieved by the contradiction argument with full use of the properties of the gradient-related sequence of descent directions and the nonmonotone step-size search rule based on the backtracking strategy.

\subsubsection{Weak convergence technique and analysis of accumulation points}
Firstly, we prove the assertion (a) in Theorem~\ref{theorem-global-convergence}, i.e., every accumulation point of the sequence $\left\{u_{n}\right\}$ is a critical point, with a weak convergence technique stated in the following lemmas.  

\begin{lemma}\label{lem:strong-weak-cvg}
For sequences $\{\ell_n\}\subset H^*$ and $\{\eta_n\}\subset H$, if $\ell_n\to\ell_*$ strongly in $H^*$ and $\eta_n\rightharpoonup \eta_*$ weakly in $H$, then $\langle\ell_n,\eta_n\rangle \to \langle\ell_*,\eta_*\rangle$ as $n\to\infty$.
\end{lemma}
\begin{proof}
Rewrite $\langle\ell_n,\eta_n\rangle=\langle\ell_*,\eta_n\rangle+\langle\ell_n-\ell_*,\eta_n\rangle$. Clearly, $\langle\ell_*,\eta_n\rangle\to\langle\ell_*,\eta_*\rangle$ since $\eta_n\rightharpoonup \eta_*$ weakly in $H$, and $\langle\ell_n-\ell_*,\eta_n\rangle\to0$ due to $\|\ell_n-\ell_*\|_{H^*}\to0$ and the fact that the weakly convergent sequence $\{\eta_n\}$ is bounded in $H$.
\end{proof}

\begin{lemma}\label{lem:strong-weak-cvg-DEhat}
Under the assumptions (A1)-(A3), let $\mathcal{R}$ be a retraction on the Nehari manifold $\mathcal{N}$. For any sequences $\{(w_n,\eta_n)\}\subset T\mathcal{N}$ and $\{s_n\}\subset\mathbb{R}$, if $w_n\to w_* \in \mathcal{N}$ strongly in $H$, $\eta_n\rightharpoonup\eta_*$ weakly in $H$ and $|s_n|\to0$ as $n\to\infty$, then $\eta_*\in T_{w_*}\mathcal{N}$ and
\begin{align}
\lim_{n\to\infty}\frac{E(\mathcal{R}_{w_n}(s_n\eta_n))-E(w_n)}{s_n}=\left\langle E'(w_*),\eta_*\right\rangle=\left(\nabla_{\mathcal{N}}E(w_*),\eta_*\right)_H. 
\end{align}
\end{lemma}
\begin{proof}
Let $\hat{\mathcal{R}}$ be the $C^1$ extension of $\mathcal{R}$ s.t. $\left.\hat{\mathcal{R}}\right|_{\mathcal{N}} = \mathcal{R}$ and define $\hat{E} := E \!\circ\! \hat{\mathcal{R}}$. Clearly, $\hat{E}$ is a $C^1$ functional on an open subset $\mathcal{U}$ of $H\times H$ containing $T\mathcal{N}$, and
\[ E(w_n)= E(\mathcal{R}_{w_n}(0)) = \hat{E}(w_n,0), \quad E(\mathcal{R}_{w_n}(s_n\eta_n)) = \hat{E}(w_n,s_n\eta_n). \]
Denote $\hat{E}'_v(u,v)\in H^*$ as the partial Fr\'echet-derivative of $\hat{E}$ w.r.t.~$v$ at $(u,v)\in\mathcal{U}$. The mean value theorem states that, for some $\tilde{s}_n\in\mathbb{R}$ between $s_n$ and $0$, 
\begin{align}\label{mean theorem}
  \frac{E(\mathcal{R}_{w_n}(s_n\eta_n))-E(w_n)}{s_n}
  =\frac{\hat{E}(w_n,s_n\eta_n)-\hat{E}(w_n,0)}{s_n}
  =\big\langle \hat{E}'_v(w_n,\tilde{s}_n\eta_n),\eta_n\big\rangle. 
\end{align}
Noticing that $|\tilde{s}_n|<|s_n|\to 0$ as $n\to\infty$ and the weakly convergent sequence $\{\eta_n\}$ is bounded in $H$, one can conclude that $(w_n,\tilde{s}_n\eta_n)\to(w_*,0)$ strongly in $H\times H$. Moreover, since $\hat{E}$ is of $C^1$, the partial Fr\'echet-derivative $\hat{E}'_v:\mathcal{U}\to H^*$ is continuous. As a result, $\hat{E}'_v(w_n,\tilde{s}_n\eta_n)\to \hat{E}'_v(w_*,0)$ strongly in $H^*$. Hence, Lemma~\ref{lem:strong-weak-cvg} yields 
\begin{align}\label{eq:Ehat-prime-limit}
  \big\langle\hat{E}'_v(w_n,\tilde{s}_n\eta_n),\eta_n\big\rangle \rightarrow \big\langle\hat{E}'_v(w_*,0),\eta_* \big\rangle \quad\mbox{as}\;\; n\rightarrow \infty. 
\end{align}
Reviewing the $C^1$ functional $G$, it holds that $G'(w_n)\to G'(w_*)$ strongly in $H^*$. Applying Lemma~\ref{lem:strong-weak-cvg} again, $\left\langle G'(w_*), \eta_* \right\rangle=\lim_{n\to\infty} \left\langle G'(w_n), \eta_n \right\rangle=0$. Thus, $(w_*,\eta_*)\in T\mathcal{N}$ and the following holds
  \begin{align}\label{eq:Ehat-prime-limit-Rg}
    \big\langle\hat{E}'_v(w_*,0),\eta_*\big\rangle 
    = \frac{d}{ds}\hat{E}(w_*,s\eta_*)\Big|_{s=0}
    = \frac{d}{ds}E(\mathcal{R}_{w_*}(s\eta_*))\Big|_{s=0}
    =\left\langle E'(w_*),\eta_*\right\rangle.
  \end{align}
The proof is finished by the combination of \eqref{mean theorem}-\eqref{eq:Ehat-prime-limit-Rg} and the definition of the Riemannian gradient in Definition \ref{def:Riemannian-gradient}.
\end{proof}

With the above preparations, we now prove (a) in Theorem~\ref{theorem-global-convergence} by the contradiction argument.

\begin{proof}[Proof of the assertion (a) in Theorem~\ref{theorem-global-convergence}]
For the sake of contradiction, suppose that there is a subsequence $\left\{u_{n}\right\}_{n \in \mathcal{I}}$ converging to a noncritical $u_{*}$, i.e., $E'(u_{*})\neq 0$. 
Since $\{\xi_n\}$ is gradient-related, the subsequence $\left\{\xi_{n}\right\}_{n \in \mathcal{I}}$ is bounded in $H$ and satisfies
\[ \delta_0:=-\limsup_{n \in \mathcal{I}}\left(\nabla_{\mathcal{N}} E\left(u_{n}\right), \xi_n \right)_H > 0. \]
By the Eberlein-Shmulyan theorem \cite{Yosida}, $\{\xi_n\}_{n \in \mathcal{I}}$ possesses a subsequence, also denoted as $\{\xi_n\}_{n \in \mathcal{I}}$ for simplicity, weakly converging to a $\xi_*\in H$. Since $E\in C^2$, by Lemma~\ref{lem:strong-weak-cvg}, one has $\left\langle G'(u_*), \xi_* \right\rangle=\lim_{n \in \mathcal{I}} \left\langle G'(u_n), \xi_n \right\rangle=0$, i.e., $\xi_*\in T_{u_*}\mathcal{N}$, and 
\begin{align}
\left\langle E'(u_*), \xi_* \right\rangle
  = \lim_{n \in \mathcal{I}}\left\langle E'(u_n), \xi_n \right\rangle
  =\lim_{n \in \mathcal{I}} \left(\nabla_{\mathcal{N}} E\left(u_{n}\right), \xi_n \right)_H  \leq -\delta_0 < 0.
\end{align}
Applying Lemma~\ref{for lemma2}, one concludes that $\sum_{n \in \mathcal{I}} \alpha_n<\infty$, and therefore, $\left\{ \alpha_n\right\}_{n \in \mathcal{I}}$ converges to $0$. Thus, for $n\in \mathcal{I}$ large enough, $\tilde{\alpha}_n:=\beta^{-1} \alpha_n= \alpha_n^0\beta^{m_n-1}>0$ does not satisfy the nonmonotone step-size rule \eqref{nonmonotone armijo step-size}, i.e., 
\begin{align}
	E\left(\mathcal{R}_{u_{n}}\left(\tilde{\alpha}_n \xi_{n}\right)\right)>C_n+\sigma\tilde{\alpha}_n\left(\nabla_{\mathcal{N}}E(u_n), \xi_n\right)_H,\quad \forall \;n\in \mathcal{I}.
\end{align}
It follows from the fact that $E(u_n)\leq C_n$ in Lemma~\ref{lem:EukCkEu0} that
  \begin{align}\label{eq:aktilde-ineq}
      \frac{E\left(\mathcal{R}_{u_{n}}\left(\tilde{\alpha}_n \xi_{n}\right)\right)-E(u_n)}{\tilde{\alpha}_n} > \sigma\left(\nabla_{\mathcal{N}}E(u_n), \xi_n\right)_H =\sigma\left\langle E'(u_n), \xi_n\right\rangle,\quad \forall \;n\in \mathcal{I}.
  \end{align}
Notice that $\tilde{\alpha}_n=\beta^{-1}\alpha_n\to0$ as $n\to\infty,\,n\in\mathcal{I}$. Taking the limit $n\to\infty,\,n\in\mathcal{I}$ in both sides in the inequality \eqref{eq:aktilde-ineq} and applying Lemmas~\ref{lem:strong-weak-cvg}-\ref{lem:strong-weak-cvg-DEhat} result in
\[ \langle E'(u_*),\xi_*\rangle \geq \sigma\langle E'(u_*),\xi_*\rangle. \]
This is a contradiction since $\sigma\in(0,1)$ and $\langle E'(u_*),\xi_*\rangle\leq-\delta_0<0$. (a) is verified. 
\end{proof}

\subsubsection{Subsequence convergence}
Here, we prove the assertion (b) in Theorem~\ref{theorem-global-convergence}, i.e., $\{u_n\}$ has a convergent subsequence, under the additional assumptions that $E$ satisfies the PS condition on the Nehari manifold $\mathcal{N}$ and the retraction $\mathcal{R}$ satisfies the condition \eqref{lipschitz continuous}. The following lemma estimates the limit behavior of the iterative step distance $\|u_{n+1}-u_n\|_H$, which is crucial to the proofs of assertions (b) and (c). 

\begin{lemma}\label{lem:ukstep-estimates}
Under the assumptions (A1)-(A3), if $\{\xi_n\}$ is strongly gradient-related and the retraction $\mathcal{R}$ satisfies \eqref{lipschitz continuous}, then
\begin{align}\label{eq:ukstep-goesto-0}
\lim_{n\to\infty}\|u_{n+1}-u_n\|_H =0,
\end{align}
and for any $\varepsilon>0$, setting $\mathcal{I}_{\varepsilon}=\{n\in\mathbb{N}: \|\nabla_{\mathcal{N}}E(u_n)\|_H \geq \varepsilon\}$, there holds
\begin{align}\label{eq:sum-ukstep-eps}
\sum_{n\in \mathcal{I}_{\varepsilon}} \|u_{n+1}-u_n\|_H <\infty.
\end{align}
\end{lemma}
\begin{proof}
The strongly gradient-related property \eqref{eq:strong-grad-related} states that
\begin{align*}
   \|\xi_n\|_H^2 \leq c_2^2\|\nabla_{\mathcal{N}}E(u_n)\|_H^2 \leq \frac{c_2^2}{c_1}|\left(\nabla_{\mathcal{N}}E(u_n),\xi_n\right)|, \quad n\geq 0.
\end{align*}
Then, Lemma~\ref{for lemma2} yields $\sum_{n=0}^{\infty} \alpha_n\|\xi_n\|_H^2 < \infty$. The uniform upper-bound for $\alpha_n$, i.e., $ \alpha_n\leq \alpha_n^0\leq\alpha_{\max}$ ($\forall n\geq 0$), leads to $\lim_{n\to\infty} \alpha_n\|\xi_n\|_H = 0$. 
By \eqref{lipschitz continuous}, for all $n$ large enough s.t. $ \alpha_n\|\xi_n\|_H <\delta$, there holds
	\begin{align*}
		\|u_{n+1}-u_n\|_H  = \|\mathcal{R}_{u_n}( \alpha_n\xi_n) - u_n\|_H\ \leq M_{\delta}  \alpha_n\|\xi_n\|_H.
	\end{align*}
Thus, \eqref{eq:ukstep-goesto-0} is verified. Further, if $\|\nabla_{\mathcal{N}}E(u_n)\|_H \geq \varepsilon$, then
	\begin{align*}
		\|u_{n+1}-u_n\|_H 
		\leq \frac{c_2M_{\delta}\alpha_n}{\varepsilon} \|\nabla_{\mathcal{N}}E(u_n)\|_H^2
        \leq \frac{c_2M_{\delta} \alpha_n}{c_1\varepsilon} \left|\left(\nabla_{\mathcal{N}}E(u_n),\xi_n\right)_H\right|,
  \end{align*}
and \eqref{eq:sum-ukstep-eps} follows straightforwardly from Lemma~\ref{for lemma2}.
\end{proof}

Based on Lemma~\ref{lem:ukstep-estimates}, the proof of subsequence convergence is given as follows. 
	\begin{proof}[Proof of the assertion (b) in Theorem~\ref{theorem-global-convergence}]
		Noting that $\{E(u_n)\}$ is bounded according to Lemma~\ref{for lemma2}, by the PS condition on the Nehari manifold $\mathcal{N}$, it suffices to verify 
		\begin{align}\label{eq:liminfG0}
		\liminf_{n\to\infty} \|\nabla_{\mathcal{N}}E(u_n)\|_H=0.
		\end{align}
  For the sake of proof by contradiction, suppose that $\|\nabla_{\mathcal{N}}E(u_n)\|_H\geq \varepsilon_0>0$ for some $\varepsilon_0>0$ and all $n>K$, where $K$ is a positive constant. Lemma~\ref{lem:ukstep-estimates} admits $\sum_{n=K+1}^{\infty}\|u_{n+1}-u_n\|_H<\infty$. Then $\{u_n\}$ is a Cauchy sequence in $H$ and converges to some $\bar{u}\in\mathcal{N}$. By the continuity, $\|\nabla_{\mathcal{N}}E(\bar{u})\|_H\geq\varepsilon_0>0$, i.e., $\bar{u}$ is not a critical point. This contradicts the assertion (a) in Theorem~\ref{theorem-global-convergence}. Thus \eqref{eq:liminfG0} holds. 
\end{proof}

\subsubsection{Whole sequence convergence}
% \subsection{Whole sequence convergence}\label{sec:proof-cvg-c}
We prove the last assertion (c) to finish the proof of Theorem~\ref{theorem-global-convergence} by establishing the convergence of the whole sequence $\{u_n\}$ to an isolated critical point $u_*$ appearing as an accumulation point. With properties \eqref{eq:ukstep-goesto-0}-\eqref{eq:sum-ukstep-eps} in Lemma~\ref{lem:ukstep-estimates} prepared, the rest of the proof of Theorem~\ref{theorem-global-convergence} can be done by the contradiction argument and using the isolation assumption of the critical point $u_*$. The basic idea is similar to that in the proof of \cite[Theorem~2.4]{Zhou2017CAMC} in which an improved convergence result for the LMM was proved.

\begin{figure}[H]
\centering
\vspace{-1ex}
\includegraphics[width=0.55\textwidth]{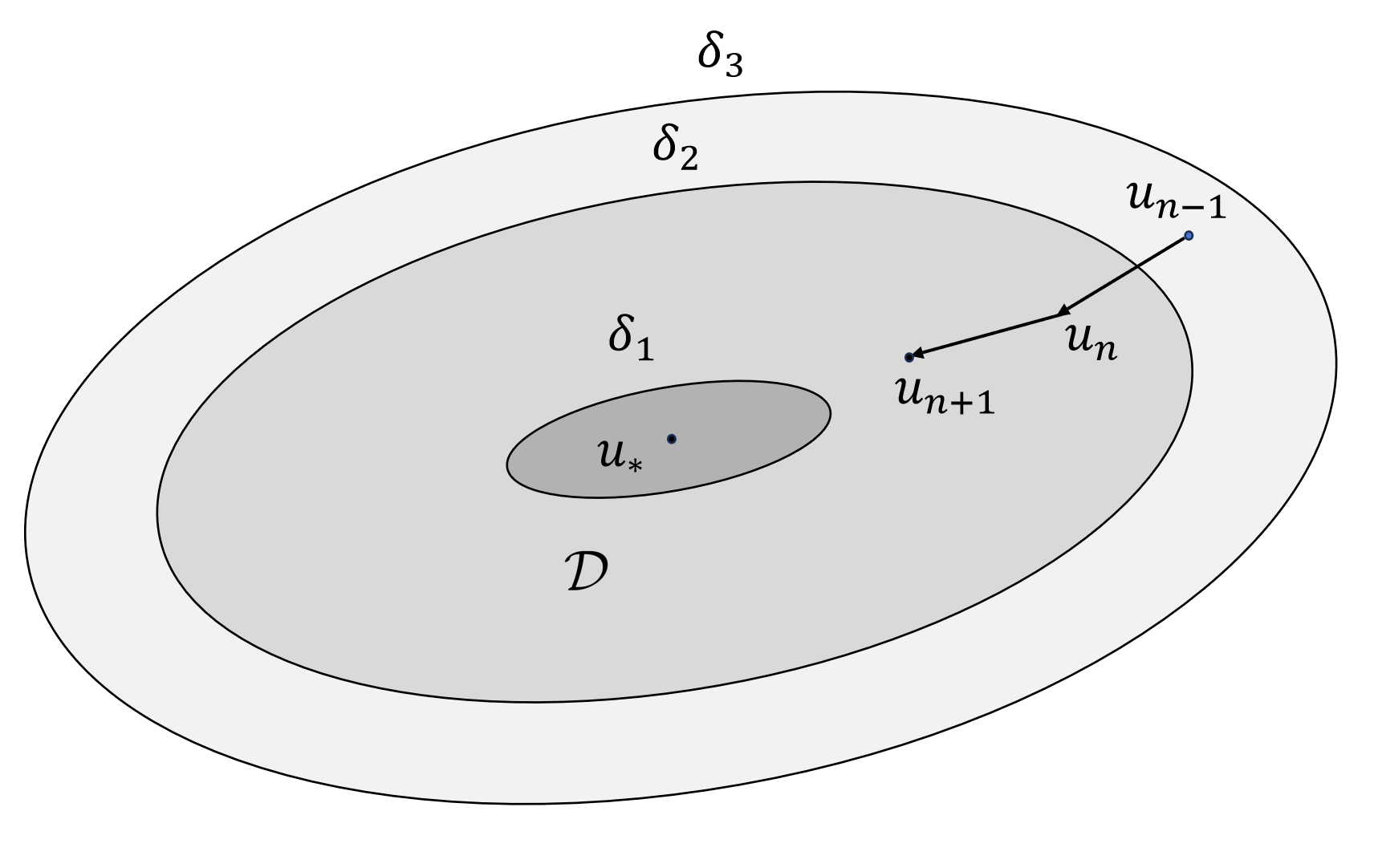}
\vspace{-1ex}
\caption{For sufficiently large $n$, $\|u_{n+1}-u_{n}\|<(\delta_2-\delta_1)/2$ and the process crossing the region $\mathcal{D}=\{u:\delta_1<\|u-u_*\|_H<\delta_2\}$ defines at least two points in $\{u_n\}$ inside $\mathcal{D}$.}
\label{fig.crossing_D}
\end{figure}

	\begin{proof}[Proof of the assertion (c) in Theorem~\ref{theorem-global-convergence}]
		What needs to be proved is that for any neighborhood of $u_*$, there are only a finite number of points in $\{u_n\}$ outside it. Suppose for the sake of contradiction that there is some constant $\delta_3>0$ s.t. $\|u -u_*\|_H >\delta_3$ holds for infinitely many points $u$ in $\{u_n\}$. For any $\delta_1,\delta_2$ satisfying $0<\delta_1<\delta_2<\delta_3$, we know that the region $\{u:\|u-u_*\|_H<\delta_1\}$ contains infinitely many points in $\{u_n\}$ since $u_*$ is an accumulation point, and that $\|u_{n+1} - u_n\|_H < (\delta_2-\delta_1)/2$ holds for all sufficiently large $n$ due to \eqref{eq:ukstep-goesto-0}. Therefore, $\{u_n\}$ travels an infinite number of times between regions $\{u:\|u-u_*\|_H<\delta_1\}$ and $\{u:\|u-u_*\|_H>\delta_3\}$ by crossing the region $\mathcal{D}=\{u:\delta_1<\|u-u_*\|_H<\delta_2\}$. As illustrated in Figure~\ref{fig.crossing_D}, when $n$ is sufficiently large, every crossing process defines at least two points in $\{u_n\}$ inside $\mathcal{D}$, and at least a distance of $(\delta_2-\delta_1)/2$ needs to be covered during this process. As a result, the index set $\mathcal{I} = \{n\in \mathbb{N}: \delta_1<\|u_n-u_*\|_H<\delta_2\}$ possesses infinitely many elements, and the total distance traveled by the subsequence $\{u_n\}_{n\in\mathcal{I}}$ is infinite, i.e., $\sum_{n\in\mathcal{I}} \|u_{n+1} -u_{n}\|_H = \infty$. From \eqref{eq:sum-ukstep-eps}, for any $\varepsilon>0$, there holds 
		\[ \sum_{n\in\mathcal{I},\|\nabla E(u_n)\|_H < \varepsilon}\|u_{n+1} -u_{n}\|_H = \infty, \] 
		and thus, there are infinitely many points $u$ in $\{u_n\}_{n\in\mathcal{I}}$ satisfying $\|\nabla_{\mathcal{N}}E(u)\|_H<\varepsilon$. Due to the arbitrariness of $\varepsilon>0$ and the PS condition on the Nehari manifold $\mathcal{N}$, the subsequence $\{u_n\}_{n\in\mathcal{I}}$ has an accumulation point $u_*'\in\mathcal{N}\cap\mathcal{D}$, which is also a critical point by the assertion (a) in Theorem~\ref{theorem-global-convergence}. This contradicts the isolation assumption for $u_*$ since $\delta_1,\delta_2$ can be chosen arbitrarily to satisfy $0<\delta_1<\delta_2<\delta_3$. Thus, the conclusion $\lim_{n\to\infty}\|u_n-u_*\|_H=0$ holds.
	\end{proof}

% \vspace{-1.5ex}
\begin{remark}
Similar to Theorem~\ref{theorem-global-convergence}, the subsequence convergence and whole sequence convergence for LMMs have also been obtained in~\cite{LMMConvergence2002,LXY-Goldstein,LXY-NWPLMM,LXY2024JCM,Zhou2017CAMC}. The analysis for LMMs is established under a two-level iterative scheme for the minimax problem of the form \eqref{minimax characters} and relies on the properties of the local peak selection. While Theorem~\ref{theorem-global-convergence} is built under an iterative framework \eqref{eq:NMOM-IterativeScheme} of manifold optimization and fully utilizes the structure of Nehari manifold and the properties of the retraction. In addition, the weak convergence technique used in this paper could avoid a certain weak continuity assumption for the general descent directions in LMMs \cite{LXY-NWPLMM}. Actually, some techniques in our proof provide a possibility to generalize certain Riemannian optimization methods to the infinite-dimensional setting.
\end{remark}
% \vspace{-1ex}

\section{Applications to semilinear elliptic PDEs}\label{Numerical experiments}
In this section, we apply the {\bf Algorithm~\ref*{alg:NMOM}$'$} to compute unstable ground states (i.e., 1-saddles of the associated variational functional) of a class of semilinear elliptic PDEs with certain variational structures. We first describe the theoretical and numerical details for a general model and then present the basic numerical results on two typical examples, including the stationary nonlinear Schr\"odinger equation (NLSE) and the H\'enon equation. Numerical comparisons of the NMOM and LMM are also provided.

% \vspace{-.8ex}
\subsection{Preliminaries}
Consider the following boundary value problem (BVP) of the semilinear elliptic PDEs
% \vspace{-.6ex}
	\begin{align}\label{our problem}
	\left\{\begin{aligned}
	-\Delta u(\mathbf{x})+a(\mathbf{x})u(\mathbf{x})-g(\mathbf{x})|u(\mathbf{x})|^{p-1}u(\mathbf{x})&=0,\quad \mbox{in}\;\; \Omega, \\
	u(\mathbf{x})&=0\quad \mbox{on}\;\; \partial\Omega.
	\end{aligned}\right.
% \vspace{-.8ex}%
	\end{align}%
Here $\Omega\subset\mathbb{R}^N$ is a bounded domain with the Lipschitz boundary $\partial\Omega$, $a\in L^{\infty}(\Omega)$ and $g\in L^{\infty}(\Omega)$ are given functions satisfying $a(\mathbf{x}) \geq a_0 > -\lambda_1$ and $g(\mathbf{x})> 0$ (a.e. $\mathbf{x}\in\Omega$) with $a_0$ a constant and $\lambda_1>0$ the first Dirichlet eigenvalue of $-\Delta $ on $\Omega$. The exponent $p$ is subcritical, i.e., $1<p<\infty$ for $N=1,2$ and $1<p<(N+2)/(N-2)$ for $N\geq3$. The existence and multiplicity of solutions to the BVP \eqref{our problem} have been widely studied, see, e.g., \cite{Rabinowitz1986,Struwe2000}. For simplicity, we will omit the variable $\mathbf{x}$ in the functions $a$, $g$, $u$, etc., in the following.

Take $H := H_0^1(\Omega)$ with the inner product and norm defined as
% \vspace{-.6ex}
\[
(u,v)_H = \int_{\Omega}\left(\nabla u\cdot \nabla v + auv\right) d\mathbf{x}, \quad
  \|u\|_{H}=\sqrt{\left(u,u\right)_H},\quad u,v\in H.
% \vspace{-.8ex}%
\]%
It is noted that $\|\cdot\|_H$ is equivalent to the usual $H_0^1$-norm $\|u\|=\left(\int_{\Omega}|\nabla u|^2 d\mathbf{x}\right)^{1/2}$ since $a\geq a_0>-\lambda_1$ is uniformly bounded. Define the functional
% \vspace{-.6ex}%
\begin{align}\label{energy}
  E(u)=\frac12\|u\|_H^2-\frac{1}{p+1}\int_{\Omega}g|u|^{p+1}d\mathbf{x},\quad u\in H.
% \vspace{-.8ex}%
\end{align}%
It is easy to check that $E \in C^2(H,\mathbb{R})$ and for all $u,v,w\in H$,
% \vspace{-.6ex}%
\begin{align}
  \langle E'(u),v\rangle=&\int_{\Omega}\left(\nabla u\cdot\nabla v+auv-g|u|^{p-1}uv\right)d\mathbf{x}, \quad \label{eq:ellip-DE} \\
  \langle E''(u)w,v\rangle=&\int_{\Omega}\left(\nabla w\cdot\nabla v+awv-pg|u|^{p-1}wv\right)d\mathbf{x}. \label{eq:ellip-DDE}
% \vspace{-.8ex}%
\end{align}%
From \eqref{eq:ellip-DE}, weak solutions of \eqref{our problem} correspond to critical points of the functional $E$ in \eqref{energy}, i.e., \eqref{energy} is the variational functional for \eqref{our problem}. Clearly, the Nehari manifold associated to the functional $E$ in \eqref{energy} is given as
\begin{equation}\label{manifold}
	\mathcal{N}=\left\{u\in H\backslash\{0\}:\|u\|_H^2
		 = \int_{\Omega}g|u|^{p+1}d\mathbf{x}\right\},
\end{equation}
with its tangent space $T_u\mathcal{N}= \left\{\xi\in H: 2(u,\xi)_H=(p+1)\int_{\Omega}g|u|^{p-1}u\xi d\mathbf{x}\right\}$ at $u\in\mathcal{N}$.

\subsection{Verification of basic assumptions and the PS condition}
It is easy to verify that the functional $E$ in \eqref{energy} satisfies the assumptions (A1)-(A3). Clearly, for each $v \in H \backslash \{0\}$, $E(sv)$ w.r.t. $s\in(0,\infty)$ has a unique critical point 
\begin{align}\label{eq:sv-expr}
  s_v=\left(\frac{\|v\|_H^2}{\int_{\Omega}g|v|^{p+1}d\mathbf{x}}\right)^{\frac{1}{p-1}}>0
\end{align}
with the fact that $s_vv\in\mathcal{N}$. Moreover, since $p>1$, one has 
\[\langle E''(u)u,u \rangle = \int_{\Omega} \left(|\nabla u|^2 +a u^2 - pg|u|^{p+1}\right) d\mathbf{x} = (1-p)\int_{\Omega} g|u|^{p+1}d\mathbf{x} <0,\quad \forall\, u \in \mathcal{N}.\]
Hence (A1) and (A2) hold. Furthermore, the following lemma (see its proof in Appendix~\ref{sec:appendix-proof-lem:E0bounds}) shows that (A3) also holds for the functional \eqref{energy}.
\begin{lemma}\label{lem:E0bounds}
For the Nehari manifold $\mathcal{N}$ \eqref{manifold}, there exists a constant $\sigma_0>0$ s.t.
\[ \|u\|_H\geq \sigma_0 , \quad \|u\|_{L^{p+1}}\geq \sigma_0,\quad \int_{\Omega}g|u|^{p+1}d\mathbf{x}\geq \sigma_0, \quad\forall\;u\in\mathcal{N}. \]
\end{lemma}

Recalling standard results in the critical point theory (see, e.g., \cite{Rabinowitz1986}), the energy functional $E$ satisfies the PS condition in $H$. The following result is sufficient for the PS condition on the Nehari manifold $\mathcal{N}$ \eqref{manifold} (see its proof in Appendix~\ref{sec:appendix-proof-lem:5.2}).  

\begin{lemma}\label{lem:5.2}
For a sequence $\{u_n\}$ on the Nehari manifold $\mathcal{N}$ \eqref{manifold}, if $\{E(u_n)\}$ is bounded and $\nabla_{\mathcal{N}} E(u_n) \to 0$ when $n\to \infty$, then $\nabla E(u_n) \to 0$ as $n\to \infty$.
\end{lemma}

In conclusion, all basic assumptions (A1)-(A3) and the PS condition on the Nehari manifold \eqref{manifold} are verified for the BVP \eqref{our problem}. Hence, the Nehari manifold $\mathcal{N}$ \eqref{manifold} is indeed a closed $C^1$ submanifold of $H$. This also establishes the mathematical justification for the existence and variational characterization of the ground state (1-saddle) to the BVP \eqref{our problem} by Theorems~\ref{theorem: 2.1} and \ref{thm:existence}. As a result, the NMOM is applicable to compute the ground state.

\subsection{Details for the numerical implementation}\label{sec:numer-details}
In our implementation of {\bf Algorithm~\ref*{alg:NMOM}$'$} to the BVP \eqref{our problem}, the steepest descent direction (i.e., the negative Riemannian gradient direction) and the Nehari retraction \eqref{Nehari retraction} is employed. Here we discuss the details for their efficient computations.

{\em (i) Computation of Riemannian gradient.}
According to \eqref{eq:Rgrad=ProjHgrad}, the Riemannian gradient $\nabla_{\mathcal{N}}E(u)\in T_{u}\mathcal{N}$ at $u\in\mathcal{N}$ can be computed by $H$-gradients $\nabla E(u)$ and $\nabla G(u)$. In fact, according to \eqref{eq:Hgrad-def} and \eqref{eq:ellip-DE}, $\varphi_1:=\nabla E(u)$ and $\varphi_2:=\nabla G(u)$ are, respectively, determined by linear problems 
\begin{align*}
\int_{\Omega}\nabla\varphi_1\cdot\nabla\phi + a\varphi_1\phi\,d\mathbf{x} &=\int_{\Omega}\left(\nabla u\cdot\nabla\phi+au\phi-g|u|^{p-1}u\phi\right)d\mathbf{x}, \;\; \forall \phi \in H,\quad\mbox{and} \\
 \int_{\Omega}\nabla\varphi_2\cdot\nabla\phi + a\varphi_2\phi\,d\mathbf{x} &=\int_{\Omega}\left(2\nabla u\cdot\nabla\phi+2au\phi-(p+1)g|u|^{p-1}u\phi\right)d\mathbf{x},\;\; \forall \phi \in H. 
\end{align*}
Then $\psi:= u-\varphi_1 = (2u-\varphi_2)/ (p+1) \in H $ solves the linear elliptic equation
\begin{align}\label{eq:poisson}
-\Delta \psi+a\psi=g|u|^{p-1}u,\quad \psi|_{\partial\Omega}=0. 
\end{align}
As a result, once the solution to \eqref{eq:poisson} is obtained, the two $H$-gradients are given as $\nabla E(u)=u-\psi$ and $\nabla G(u)=2u-(p+1)\psi$. Thus, for the BVP \eqref{our problem}, each computation of the Riemannian gradient only requires to solve one linear elliptic equation.

{\em (ii) Explicit Nehari retraction.}
For $v\in H\backslash\{0\}$, $\rho(v)$ in the Nehari retraction \eqref{Nehari retraction} is expressed as the unique positive solution of the equation $\langle E'(sv),v\rangle=0$ w.r.t.~$s$, i.e., $\rho(v)=s_v$, given in \eqref{eq:sv-expr}. Thus, the Nehari retraction \eqref{Nehari retraction} for the problem \eqref{our problem} takes an explicit formulation for $(u,\xi)\in T\mathcal{N}$:
\begin{align}\label{eq:NehariR}
\mathcal{R}_{u}(\xi)=\rho(u+\xi)(u+\xi)
=\left(\frac{\|u+\xi\|^2_H}{\int_{\Omega}g|u+\xi|^{p+1}d\mathbf{x}}\right)^{\frac{1}{p-1}}(u+\xi).
\end{align}
Further, we have the following property (see its proof in Appendix~\ref{sec:appendix-proof-lem:5.3}).  	
\begin{lemma}\label{lem:5.3}
	The Nehari retraction \eqref{eq:NehariR} satisfies the condition \eqref{lipschitz continuous}. 
\end{lemma}

With the fact that the steepest descent direction is strongly gradient-related as noted in Remark~\ref{re:strongg4.1}, Lemma~\ref{lem:5.3} shows that the global convergence theory described in Section~\ref{sec:Convergence} is guaranteed in our numerical computations.

{\em (iii) Algorithm parameters.}
Unless specified, the parameters in {\bf Algorithm~\ref*{alg:NMOM}$'$} are chosen as follows:
\[ \sigma =10^{-3}, \quad
\beta = 0.25, \quad
\varrho = 0.85, \quad
\alpha_{\min} = 1, \quad
\alpha_{\max} = 10, \quad
\varepsilon_{\mathrm{tol}} = 10^{-6}. \]
The trial step-size $\alpha_n^0$ in {\bf Algorithm~\ref*{alg:NMOM}$'$} is computed by 
\begin{align}\label{eq:ak0-BB}
     \alpha_n^0 = 
    \begin{cases}
    \tau_0, & \mbox{if}\;\; n=0, \\
    \min\{\max\{\alpha_{\min}, \tau_{n}^{(1)}\},\alpha_{\max}\}, & \mbox{if}\;\; n\geq1 \text{ is even}, \\  
     \min\{\max\{\alpha_{\min}, \tau_{n}^{(2)}\},\alpha_{\max}\}, & \mbox{otherwise},
    \end{cases}
\end{align}
where $\tau_0=1$, and $\tau_{n}^{(1)}$ as well as $\tau_{n}^{(2)}$ are Barzilai-Borwein type step-sizes \cite{2017AdaptiveBB,IP2018IMANUM,LXY2024JCM}:
\begin{equation}
	\tau_{n}^{(1)}=\frac{(w_{n-1},w_{n-1})_H}{\left|(w_{n-1},y_{n-1})_H\right|}, \quad \tau_{n}^{(2)}=\frac{\left|(w_{n-1},y_{n-1})_H\right|}{(y_{n-1},y_{n-1})_H},
\end{equation}
with $w_{n-1} := u_{n} - u_{n-1}$ and $y_{n-1} := \nabla_{\mathcal{N}}E(u_{n}) - \nabla_{\mathcal{N}}E(u_{n-1})$ ($n\geq1$). In addition, the initial $u_0=\rho(\tilde{v}_0)\tilde{v}_0\in\mathcal{N}$ with $\rho(\tilde{v}_0)=s_{\tilde{v}_0}$ given in \eqref{eq:sv-expr} and $ \tilde{v}_0$ chosen as
\begin{align}\label{eq:initial-v0}
\tilde{v}_0(\mathbf{x}) = 
    \begin{cases}
        (x-1)^2(x+1), & \mbox{if } \mathbf{x}=x\in\Omega = (-1,1), \\
        (1-x^2)(1-y^2), & \mbox{if }\mathbf{x}=(x,y)\in\Omega = (-1,1)^2, \\
        (1-x^2-y^2)e^{2x+y},& \mbox{if } \mathbf{x}=(x,y)\in\Omega = \{(x,y):x^2+y^2<1\}.    
    \end{cases}
\end{align}
The linear subproblem \eqref{eq:poisson} is solved by the finite element and spectral-Galerkin methods. For the finite element discretization, 1000 uniform elements are adopted for the 1D domain $\Omega = (-1,1)$, and $100\times 100$ uniform rectangular elements are adopted for the 2D domain $\Omega = (-1,1)^2$. While, for the spectral discretization in the 2D domain $\Omega = \{(x,y):x^2+y^2<1\}$, 64 Fourier basis functions and 64 Legendre-Lobatto basis functions are used, respectively, in the polar angular and the radial directions \cite{Shen1997SISC}.

\subsection{Numerical test}
In this subsection, we show numerical results of {\bf Algorithm~\ref*{alg:NMOM}$'$} for computing the ground states of two special cases of the BVP~\eqref{our problem}: 
\begin{enumerate}[(i)]
\item The {\em stationary NLSE} (by setting $g({\bf x})=1$ and $a({\bf x})=\omega|\mathbf{x}|^{2}+\lambda$ with $\omega>0$ and $\lambda>-\lambda_1$ two positive parameters in \eqref{our problem}), i.e.,
	\begin{align}\label{nonlinear Schrodinger equation}
	\left\{\begin{aligned}
	-\Delta u(\mathbf{x})+\left(\omega|\mathbf{x}|^{2}+\lambda\right) u(\mathbf{x}) &=u^{3}(\mathbf{x}) \quad \mbox{in}\;\; \Omega, \\
	u(\mathbf{x})&=0\qquad\;\;\; \mbox{on}\;\; \partial\Omega.
	\end{aligned}\right.
	\end{align}
\item {The \em H\'enon equation} (by setting $a({\bf x})=0$ and $g({\bf x})=|\mathbf{x}|^{l}$ with $l>0$ a positive parameter in \eqref{our problem}), i.e.,
	\begin{align}\label{eq:henon}
	\left\{\begin{aligned}
	\Delta u(\mathbf{x})+|\mathbf{x}|^{l}|u(\mathbf{x})|^{p-1} u(\mathbf{x}) &=0 \quad \mbox{in}\;\; \Omega, \\
	u(\mathbf{x})&=0\quad \mbox{on}\;\; \partial\Omega.
	\end{aligned}\right.
	\end{align}
\end{enumerate}
For the H\'enon equation, $\Omega$ is assumed to be radially symmetric unless specified. Then, some numerical comparisons with the LMM are given in Section~\ref{sec:comparison}.

\subsubsection{Ground states of the stationary NLSE and the H\'enon equation} \label{sec:gro_stat}
The profiles of the ground states of the stationary NLSE \eqref{nonlinear Schrodinger equation} and the H\'enon equation \eqref{eq:henon} in 1D/2D solved by {\bf Algorithm~\ref*{alg:NMOM}$'$} are shown in Figures~\ref{fig.NLS1}-\ref{fig.henon symmertry on unit circle}. From Figures ~\ref{fig.NLS1}-\ref{fig.henon symmertry on unit circle} and additional results not shown here, we can draw the observations: For the NLSE~\eqref{nonlinear Schrodinger equation}, when $\omega$ or $\lambda$ increases, the ground state of it exhibits a progressively steeper profile, both the peak value and the energy increasing (see Figure~\ref{fig.NLS1}); For the H\'enon equation \eqref{eq:henon} with fixed $p$, both the peak value and the energy of the ground state of it increase w.r.t. $l$ (see Figures~\ref{fig.1} and \ref{fig.henon symmertry on unit circle}). More interestingly, the peak position of the ground state of the H\'enon equation \eqref{eq:henon} moves from the center to the boundary, i.e., ground states change from radially symmetric to non-radially symmetric. Later in Section~\ref{sec:symmetry-breaking}, we further investigate such symmetry-breaking phenomenon of the H\'enon equation \eqref{eq:henon}.

	\begin{figure}[!t]
		\centering
  \vspace{-1ex}
		\subfloat{\includegraphics[width=0.32\textwidth,height=0.28\textwidth]{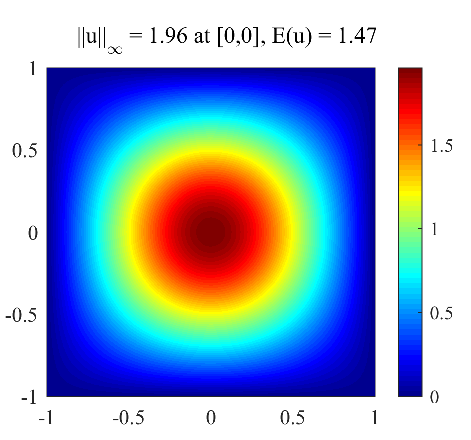}}
		\subfloat{\includegraphics[width=0.32\textwidth,height=0.28\textwidth]{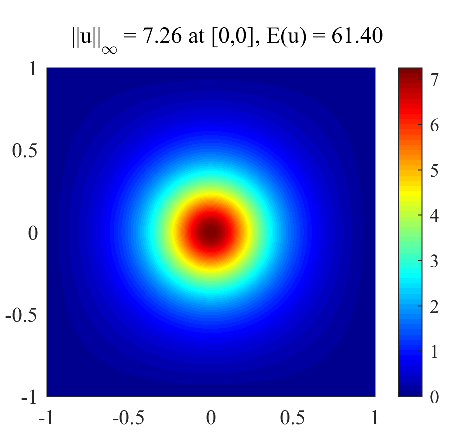}}
		\subfloat{\includegraphics[width=0.32\textwidth,height=0.28\textwidth]{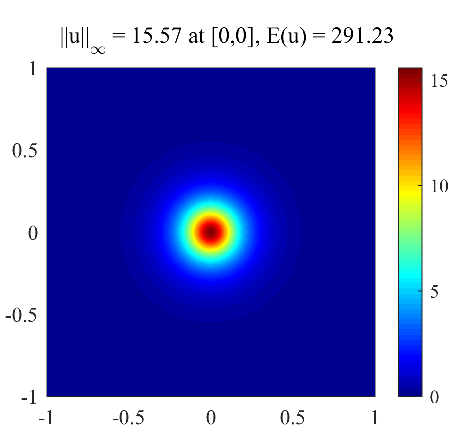}} \\[-2ex]
		\subfloat{\includegraphics[width=0.32\textwidth,height=0.28\textwidth]{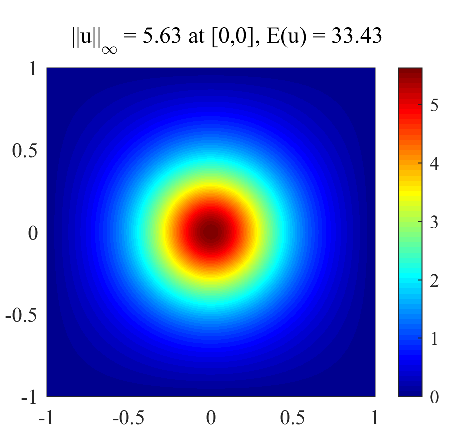}}	
		\subfloat{\includegraphics[width=0.32\textwidth,height=0.28\textwidth]{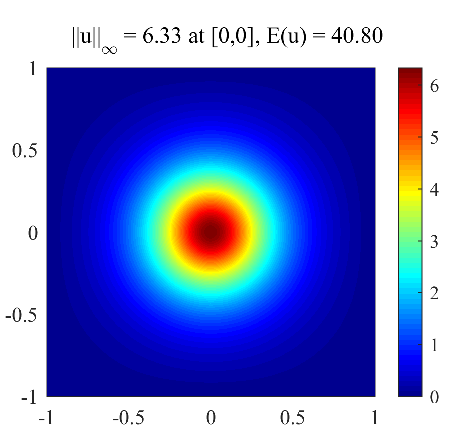}}
		\subfloat{\includegraphics[width=0.32\textwidth,height=0.28\textwidth]{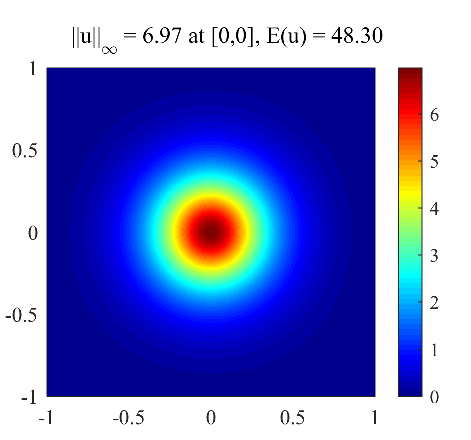}}	\\[-1ex]
		\caption{Profiles of ground states of the stationary NLSE \eqref{nonlinear Schrodinger equation} with different values of $\lambda$ and $\omega$ in $\Omega =(-1,1)^2$. The first line shows the profiles when $\omega=4$ and $\lambda=-4,10,50$ (from the left to the right). The second line shows the profiles when $\lambda = 4$ and $\omega = 10,25,45$ (from the left to the right).}
		\label{fig.NLS1}
	\end{figure}
 
\begin{figure}[!t]
		\centering
  \vspace{-1ex}
		\includegraphics[width=0.48\textwidth,height=.38\textwidth]{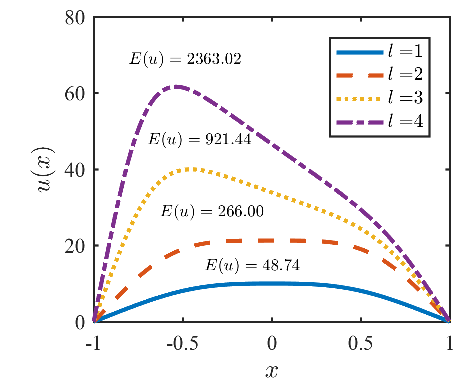} \;
		\includegraphics[width=0.48\textwidth,height=.38\textwidth]{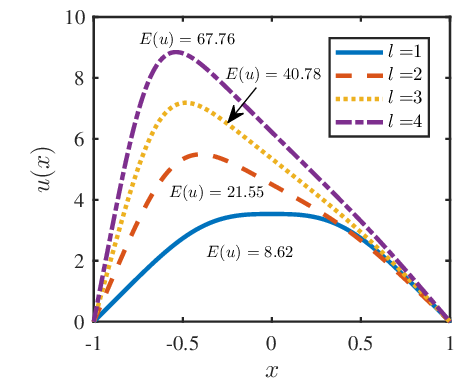} 
  \vspace{-2ex}
		\caption{Profiles of ground states of the H\'enon equation \eqref{eq:henon} in $\Omega =(-1,1)$ with different $l$ when $p=2$ (left) and $p=3$ (right).}
		\label{fig.1} 
\end{figure}

\begin{figure}[!t]
		\centering
  \vspace{-1ex}
		\subfloat{\includegraphics[width=0.32\textwidth]{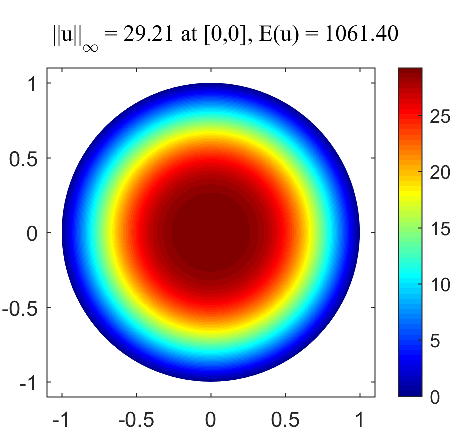}}
		\subfloat{\includegraphics[width=0.32\textwidth]{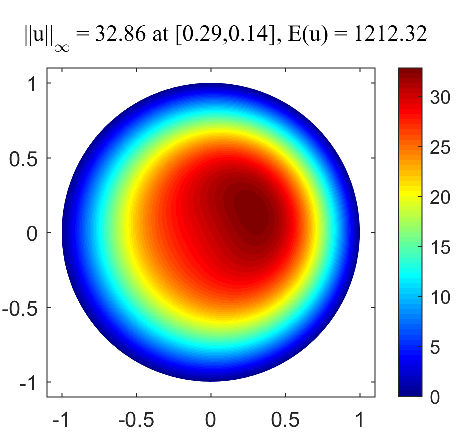}}
		\subfloat{\includegraphics[width=0.32\textwidth]{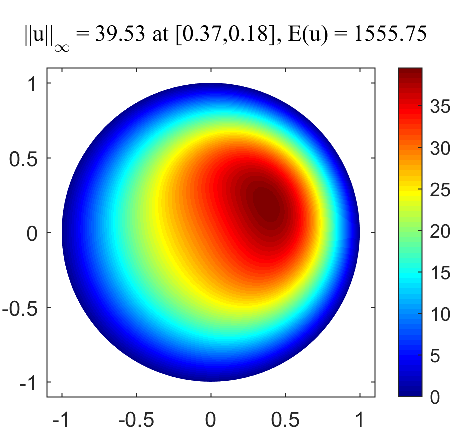}} \\[-1ex]
		\subfloat{\includegraphics[width=0.32\textwidth]{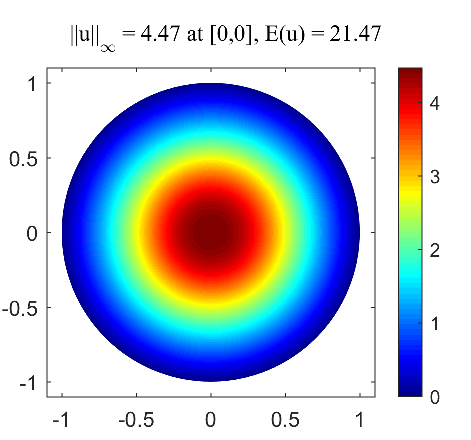}}
		\subfloat{\includegraphics[width=0.32\textwidth]{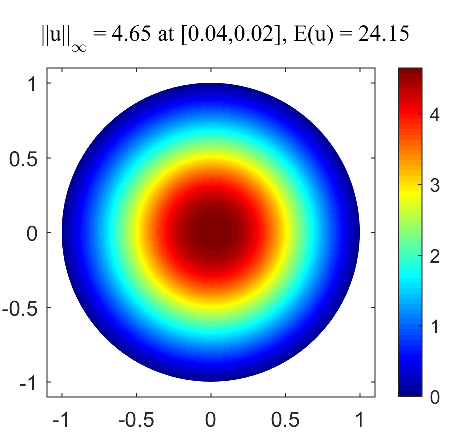}}
		\subfloat{\includegraphics[width=0.32\textwidth]{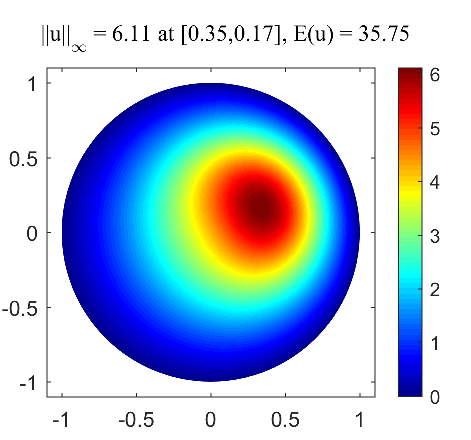}} \\[-1ex]
		\caption{Profiles of ground states of the H\'enon equation \eqref{eq:henon} with different values of $l$ and $p$ in $\Omega =\{(x,y):x^2+y^2<1\}$. The first line shows the profiles when $p=2$ and $l=1.7,1.8,2$ (from the left to the right). The second line shows the profiles when $p=3$ and $l=0.5,0.6,1$ (from the left to the right).}
        % \vspace{-10pt}
		\label{fig.henon symmertry on unit circle} 
\end{figure}

\subsubsection{Comparisons with the  LMM}\label{sec:comparison}
In this subsection, the proposed NMOM is compared with the LMM for computing ground states of the H\'enon equation \eqref{eq:henon} when $\Omega = (-1,1)^2$. Recall that the iteration scheme of the LMM for finding 1-saddles can be formulated as \cite{LMM,LMMConvergence2002,LXY-Goldstein,LXY-NWPLMM,LXY2024JCM,XieYuanZhou2012}
\begin{align}\label{eq:up_LMM}
    % \begin{cases}
        u_{n} = \rho(v_n)v_n\equiv\arg\max_{u\in\{sv_n:s>0\}} E(u),\qquad 
        v_{n+1} = \frac{v_n - \alpha_n \nabla E(u_n)}{\|v_n - \alpha_n \nabla E(u_n)\|_H},
\end{align}
with a given initial unit direction $v_0\in S_H$, where $\rho(v_n)=s_{v_n}$ is given in \eqref{eq:sv-expr}.

We first compare {\bf Algorithm~\ref*{alg:NMOM}$'$} with a nonmonotone LMM proposed in \cite{LXY2024JCM}, i.e., the LMM \eqref{eq:up_LMM} with the step-size $\alpha_n = \alpha_n^0 \beta^m$ ($m\in\mathbb{N}$) determined by the backtracking strategy for the condition
\begin{align}\label{eq:non_LMM}
E\left(\rho(v_{n+1})v_{n+1}\right) \leq C_k - \sigma \alpha_n \rho(v_n) \|\nabla E(u_n)\|_H^2,
\end{align}
with $C_k$ iteratively computed according to \eqref{eq:CkQk-update}. 

In the computation, the parameters of the nonmonotone LMM~\eqref{eq:up_LMM}-\eqref{eq:non_LMM} are consistent with those of {\bf Algorithm~\ref*{alg:NMOM}$'$} given in Section~\ref{sec:numer-details}, and the linear problem~\eqref{eq:poisson} is solved by the sine pseudospectral method with the mesh size $1/32$. The initial values for the nonmonotone LMM~\eqref{eq:up_LMM}-\eqref{eq:non_LMM} and {\bf Algorithm~\ref*{alg:NMOM}$'$} are taken as $v_0 = \tilde{v}_0/\|\tilde{v_0}\|_H$ and $u_0  = \rho(\tilde{v}_0) \tilde{v}_0$, respectively, where
\[\tilde{v}_0 = (1-x^2)(1-y^2)\left(2(x-0.5)^2 + (y+0.5)^2 \right), \]
and the stopping criterion for iterations is $\| \Delta u_n(\mathbf{x}) + |\mathbf{x}|^l |u_n(\mathbf{x})|^{p-1} u_n(\mathbf{x})\|_{\infty} < 10^{-4}$. Under different trial step-size $\alpha_n^0$, the number of iterations and the CPU time of {\bf Algorithm~\ref*{alg:NMOM}$'$} and  nonmonotone LMM~\eqref{eq:up_LMM}-\eqref{eq:non_LMM} are listed in Table~\ref{tab:eff_comp1}, from which we can draw the following observations:
\begin{itemize}
    \item When the trial step size $\alpha_n^0 = 1$, the NMOM obviously performs better in our computations, requiring fewer iterations and less CPU time than the nonmonotone LMM~\eqref{eq:up_LMM}-\eqref{eq:non_LMM} in most cases (see Table~\ref{tab:eff_comp1}).  
    \item When the trial step size is reduced to $\alpha_n^0 = 0.1$, the nonmonotone LMM~\eqref{eq:up_LMM}-\eqref{eq:non_LMM} usually outperforms the NMOM, while still cannot comparable to the NMOM under $\alpha_n^0 = 1$.   
\end{itemize} 
These results suggest that the NMOM with a suitable trial step-size could be  more  efficient than the LMM in computing the ground states of the H\'enon equation.

In addition, we compare the performance of  NMOM (Algorithm~\ref{alg:NMOM}) and  LMM iteration~\eqref{eq:up_LMM} under different fixed step-sizes (without nonmonotone search strategy). The results are listed in Table~\ref{tab:eff_comp2}. NMOM exhibits lower sensitivity to the step-size selection, successfully converging even with large step-sizes (e.g., $\alpha_n = 1$) for computing different solutions. While LMM requires suitably small step-sizes for convergence. Furthermore, it is observed that the `maximal' allowed step-sizes for the LMM decrease as the norm of the ground states $\|u_*\|_H$ increase.

\begin{table}[!t]
\centering\footnotesize
\caption{Comparison of the {\bf Algorithm~\ref*{alg:NMOM}$'$} (NMOM) and the nonmonotone LMM (nmLMM)~\eqref{eq:up_LMM}-\eqref{eq:non_LMM} in terms of the iteration count (iter) and the CPU time (time, in seconds) required to compute ground states of the H\'enon equation~\eqref{eq:henon} when $\Omega = (-1,1)^2$.}
\label{tab:eff_comp1}
\begin{tabular}{rrrrrrlrrrr}
\hline 
 \multicolumn{1}{c}{\multirow{4}{*}{$p$}} & \multicolumn{1}{c}{\multirow{4}{*}{$l$}} &
  \multicolumn{4}{c}{\multirow{2}{*}{$\alpha_n^0 = 1$}} & \multicolumn{1}{c}{\multirow{2}{*}{}} & \multicolumn{4}{c}{\multirow{2}{*}{$\alpha_n^0 = 0.1$}} \\
  \rule{0pt}{0.5em}
    &     & \multicolumn{4}{l}{}  &            & \multicolumn{4}{l}{}      
  \\ \cline{3-6} \cline{8-11}
  \rule{0pt}{1.5em} 
 &
   &
  \multicolumn{2}{c}{\multirow{1}{*} {NMOM}} &
  \multicolumn{2}{c}{\multirow{1}{*} {nmLMM}} &
  & 
  \multicolumn{2}{c}{\multirow{1}{*} {NMOM}} &
  \multicolumn{2}{c}{\multirow{1}{*} {nmLMM}}  \\ 
    &     & iter & time     & iter &  time   &  & iter & time     & iter &  time \\ \hline
\rule{0pt}{1.1em} 
1.5 & 0.5 & 39   & 0.1195   & 271  & 2.3169  &  & 450  & 1.1727   & 258  & 1.2332   \\
1.5 & 1   & 51   & 0.1485   & 291  & 2.6927  &  & 574  & 1.4504   & 381  & 2.8704   \\
2   & 1   & 95   & 0.2596   & 339  & 2.0890  &  & 996  & 2.6034   & 207  & 0.5986   \\
2   & 1.5 & 530  & 1.4369   & 273  & 1.7176  &  & 5343 & 13.8968  & 252  & 0.7948   \\
2.5 & 1.5 & 63   & 0.1981   & 205  & 1.0973  &  & 659  & 1.8261   & 466  & 1.2351   \\
2.5 & 2   & 55   & 0.1854   & 227  & 1.2851  &  & 589  & 1.6497   & 329  & 0.9757   \\
3   & 2   & 62   & 0.1706   & 221  & 1.0340  &  & 661  & 1.6499   & 45   & 0.1019   \\
3   & 2.5 & 65   & 0.1725   & 221  & 1.0561  &  & 694  & 1.7456   & 44   & 0.0985   \\ \hline 
\end{tabular}
\end{table}

\begin{table}[!t]
\centering\footnotesize
\caption{Comparison of the NMOM and LMM under fixed step-sizes in terms of the  iteration count (iter) and CPU time(time, in seconds) required to compute the ground states of~\eqref{eq:henon} when $\Omega = (-1,1)^2$. The notation `\text{---}' means the convergence failure.}
\label{tab:eff_comp2}
\begin{tabular}{llrrlrrrrr}
\hline
\multicolumn{1}{c}{\multirow{2}{*}{$p$}} & \multicolumn{1}{c}{\multirow{2}{*}{$l$}} & \multicolumn{1}{c}{\multirow{2}{*}{$\|u_*\|_H$}} & \multicolumn{1}{c}{\multirow{2}{*}{$\alpha_n$}} &  & \multicolumn{2}{c}{\multirow{1}{*}{NMOM}} & & \multicolumn{2}{c}{\multirow{1}{*}{LMM}} \\ \cline{6-7} \cline{9-10}
 \rule{0pt}{1.2em} 
 & & &  &  & iter  & time &  &iter & time  \\ \hline
\rule{0pt}{1.1em}
\multirow{4}{*}{1.5} & \multirow{4}{*}{0.5} & \multirow{4}{*}{190.3025} & 1 &  & 39 & 0.1249 &  &\multirow{1}{*}{\text{---} } &\multirow{1}{*}{\text{---} }\\
                     &            &          & 0.1 &  & 450 & 1.0986 &  & \multirow{1}{*}{\text{---} } & \multirow{1}{*}{\text{---} }\\
                     &            &          & 0.01 &  & 4551 & 11.5752 &  &  \multirow{1}{*}{\text{---} } & \multirow{1}{*}{\text{---} }\\
                     &             &          & 0.001 &  & 45562 & 115.7487 &  &  199 & 0.4861 \\ \hline
\rule{0pt}{1.1em}
\multirow{4}{*}{2}   & \multirow{4}{*}{1} & \multirow{4}{*}{37.3289}  & 1 &  & 95 & 0.2338 &  &  \multirow{1}{*}{\text{---} } & \multirow{1}{*}{\text{---} } \\
                     &          &            & 0.1 &  & 996 & 3.0815 &  &  \multirow{1}{*}{\text{---} } & \multirow{1}{*}{\text{---} }\\
                     &           &           & 0.01&  & 10000 & 26.7883 &  & 258 & 0.5876  \\
                     &           &           & 0.001 &  & 100047 & 257.8133&  &  2624 & 6.0588\\ \hline 
\rule{0pt}{1.1em}
\multirow{4}{*}{4}   & \multirow{4}{*}{3}  & \multirow{4}{*}{8.0505}  &1    &  & 111    & 0.2828   &  & \multirow{1}{*}{\text{---} }& \multirow{1}{*}{\text{---} }\\
                     &          &            &0.1  &  & 1145   & 2.8548   &  & 138    & 0.2966   \\
                     &          &            &0.01 &  & 11485  & 28.5428  &  & 1411   & 3.4784   \\
                     &          &            &0.001&  & 114886& 332.3337&  & 14133  & 31.6023  \\ \hline
\end{tabular}
% \vspace{-1ex}
\end{table}

\section{Symmetry-breaking phenomenon of the H\'enon equation}\label{sec:symmetry-breaking} 
The H\'enon equation \eqref{eq:henon} was introduced in 1973 by H\'enon when he studied the structure of rotating stars \cite{Henon1973AA} and then attracted extensive attentions. It is a model for the mass distribution in the spherical symmetric star cluster, where $\mathbf{x}$ is the location of the star, $u$ denotes the stellar density, and parameters $p$ and $l$ represent respectively the central density of each stellar type and the ratio of the transverse velocity to the radial velocity \cite{2021Numerical}. 

Existing researches show that the symmetry of the ground state is related to the parameters $p$ and $l$. Inspired by the interesting numerical observations in \cite{2000Algorithms} on the symmetry-breaking phenomenon, it is proved in \cite{2002NON} that if $\Omega$ is a $N$-dimensional unit ball with $N\geq 2$, for any subcritical exponent $p$, there exists a symmetry-breaking critical value $l^*(p)>0$ s.t. the ground state is no longer symmetric provided $l>l^*(p)$.  For the case where $\Omega=(-1,1)$ and $1<p<\infty$, the theoretical result $0<l^*(p)<4/(p-1)$ is presented in \cite{WOS:000457113400007}. While in numerical aspects, the symmetry-breaking critical value is numerically estimated as $l^*(p)|_{p=3}= 0.5886933$  for $\Omega = (-1/2,1/2)^2$ in \cite{2008Yang} and $l^*(p)|_{p=3}\in[1.2, 1.25]$ for $\Omega = (-1/2,1/2)$ in \cite{2021Numerical}.  However, the exact value of the symmetry-breaking critical value $l^*(p)$ has not yet been given in theory. In addition, until now, no approximate results for the dependency of the symmetry-breaking critical value $l^*(p)$ w.r.t. $p$ has been shown.

In this section, we concentrate on the symmetry-breaking phenomenon of ground states of the H\'enon equation \eqref{eq:henon} when $\Omega = (-1,1)$ and $\Omega = \{(x,y):x^2+y^2<1\}$, and aim to numerically study the relationship of the symmetry-breaking critical value $l^*(p)$ and the exponent $p$.  The ground states are computed by {\bf Algorithm~\ref*{alg:NMOM}$'$} following the parameter settings in Section~\ref{sec:numer-details}. Note that, if $u(\mathbf{x})$ solves \eqref{eq:henon}, then for any $R>0$, $R^{(l+2)/(p-1)}u(R\mathbf{x})$ is exactly the solution to the H\'enon equation \eqref{eq:henon} on the domain $R^{-1}\Omega$. Thus, the radius does not affect the symmetry of the ground state. Our computations use a non-radial initial guess given in \eqref{eq:initial-v0} in order to preferentially find the non-radially symmetric ground states (if exist). Through extensive computations, we can observe and conclude the following numerical findings.%

\subsection{One-dimensional case}
Figure~\ref{fig.1} shows that the symmetry-breaking points $\left.l^*(p)\right|_{p=2} \in (2,3)$ and $\left.l^*(p)\right|_{p=3} \in (1,2)$ when $\Omega = (-1,1)$. In order to numerically find the symmetry-breaking critical value $l^*(p)$ for various $p$, we introduce the asymmetry degree $\mu(u)$ for a nonzero function $u(x)$ to measure its radial symmetry in $\Omega = (-1,1)$ as
\[\mu(u):=\frac{\max _{x \in \Omega}(|u(x)-u(-x)|)}{\max _{x \in \Omega}|u(x)|}.\]
Obviously, $\mu(u) = 0$ if $u$ is radially symmetric (i.e., even) and $\mu(u)>0$ otherwise, and a larger value of $\mu$ means stronger asymmetry. Numerically, the computed ground state is considered to be radially symmetric with $\mu < 10^{-5}$. 

Fix $p$, and use the bisection method to narrow the interval of $l$ where the symmetry-breaking point exists. The numerically computed symmetry-breaking points $(p_j,\tilde{l}^*_j(p_j))$ for $p_j = 1.4+0.05j$, $j=1,2,\ldots,n_p=171$, are obtained and plotted in the left of Figure~\ref{fig.one-dimensional symmertry-breaking points}, which shows that the numerical symmetry-breaking critical value $\tilde{l}^*(p)$ is inversely proportional to $p-1$. Set $\hat{l}^*(p_j):=k_0/(p_j-1)$ to be an approximation to the numerical symmetry-breaking critical value $\tilde{l}^*(p)$, with the coefficient $k_0>0$ to be determined by the data fitting. In fact, minimizing the mean square error $\mathrm{MSE}=\frac{1}{n_p}\sum_{j=1}^{n_p}\big[\tilde{l}_j^*(p_j) - \hat{l}^*(p_j)\big]^2$ results in $k_0=2.4671$ and 
\begin{align}\label{fitting}
  \hat{l}^*(p) = \frac{k_0}{p-1}=\frac{2.4671}{p-1} \quad \mbox{with}\quad \mathrm{MSE}=3.5778\times 10^{-4}.
\end{align}
Interestingly, the coefficient $k_0=2.4671$ in $\hat{l}^*(p)$ \eqref{fitting} is very close to the first Dirichlet eigenvalue $\lambda_1=\pi^2/4\approx 2.4674$ of $-\Delta$ in $\Omega=(-1,1)$ with the error $|\lambda_1-2.4671|=3.01\times 10^{-4}$. In addition, $\hat{l}^*(p)\big|_{p=3} \in [1.2,1.25]$ and $\hat{l}^*(p)<4/(p-1)$, agreeing with numerical observations in \cite{2021Numerical} and theoretical estimates of a upper bound in \cite{WOS:000457113400007}, respectively.

\subsection{Two-dimensional case}
One can observe from Figure \ref{fig.henon symmertry on unit circle} that, when $\Omega = \{(x,y):x^2+y^2 <1\}$, the symmetry-breaking point $\left. l^*(p) \right|_{p=2} \in (1.7,1.8)$ and $\left. l^*(p) \right|_{p=3} \in (0.5,0.6)$. Similar as the 1D case, a nonzero function $u(x,y) = \tilde{u}(r,\theta)$ in the unit disk, with $(r,\theta)$ the polar coordinates, is numerically considered to be radially symmetric if its asymmetry degree $\mu(u)<10^{-4}$, where
\begin{align*}
\mu(u) := \frac{\max_{0\leq r\leq 1}\left(|\max_{0\leq\theta \leq 2\pi}\tilde{u}(r,\theta) - \min_{0\leq\theta \leq 2\pi}\tilde{u}(r,\theta)|\right)}{\max_{0\leq r \leq 1,0\leq \theta \leq 2\pi}\left|\tilde{u}(r,\theta)\right|}.
\end{align*}
By the bisection method, the numerical symmetry-breaking points $(p_j,\tilde{l}^*_j(p_j))$, $p_j = 1.4+0.1j$, $j=1,2,\ldots,n_p=76$, are obtained and  plotted in the left of Figure~\ref{sym_breaking_2dim}. It shows that $l^*(p)$ decreases as $p$ increases. However, there is no obvious linear inverse dependence between $l^*(p)$ and $p-1$, which is distinguished from the 1D case. In fact, the middle of Figure~\ref{sym_breaking_2dim} suggests that $(p-1)l^*(p)$ decreases exponentially  w.r.t.~$p$. Inspired by this, fitting $\tilde{l}^*(p)$ by $\hat{l}^*(p) ={k_1 k_2^{-p}}/{(p-1)}$ with parameters $k_1,k_2>0$ in the least-squares sense, one can arrive at (see the right of Figure~\ref{sym_breaking_2dim}):
\begin{align*}
\hat{l}^*(p) = \frac{3.69\times1.46^{-p}}{p-1}\quad \mbox{with}\quad \mathrm{MSE}=1.33\times 10^{-5}.
\end{align*}

\begin{figure}[!t]
		\centering
  \vspace{-1ex}
	\subfloat{\label{fig.l_p_times_p-1}\includegraphics[width=0.34\textwidth,height=.28\textwidth]{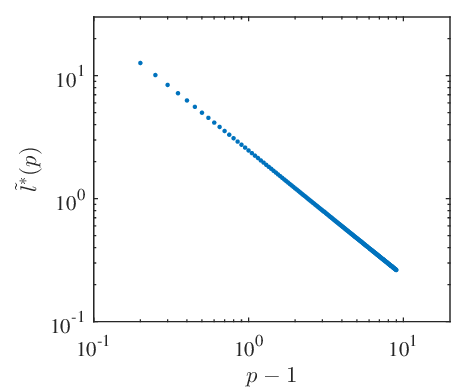}} \quad
 \subfloat{\label{fig.symmetry-breaking}\includegraphics[width=0.34\textwidth,height=.28\textwidth]{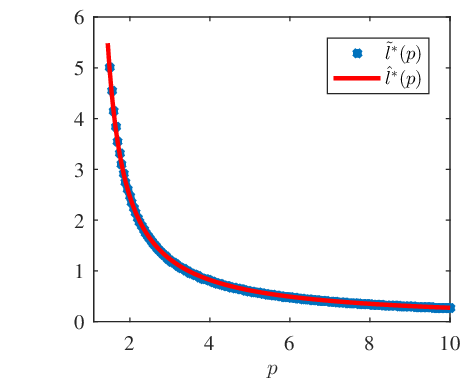} }
 \vspace{-1ex}
	\caption{Scatter plot of numerical symmetry-breaking points $(p_j-1,\tilde{l}_j^*(p_j))$ when $\Omega = (-1,1)$ in logarithmic scales (left) and the corresponding numerical fitting curve $\hat{l}^*(p) = 2.4671/(p-1)$ (right).}
  \label{fig.one-dimensional symmertry-breaking points}
	\end{figure}

\begin{figure}[!t]
\centering
  \includegraphics[width=0.33\textwidth,height=.28\textwidth]{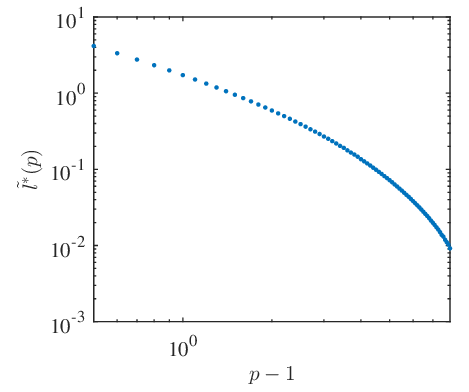} \!\!\!
  \includegraphics[width=0.33\textwidth,height=.28\textwidth]{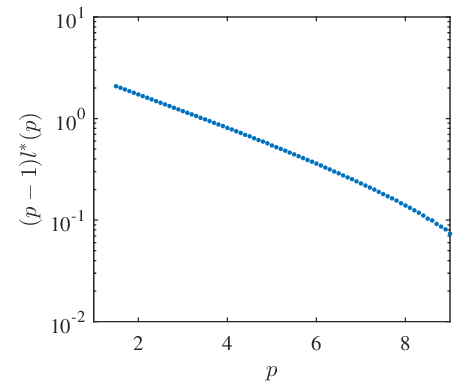} \!\!\!
  \includegraphics[width=0.33\textwidth,height=.28\textwidth]{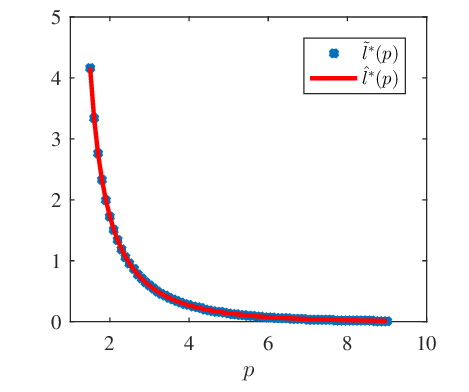} \\[-1ex]
  \caption{Scatter plot of numerical symmetry-breaking points $(p_j-1,\tilde{l}_j^*(p))$ (left) and $(p_j-1,(p_j-1)\tilde{l}_j^*(p) )$ (middle) when $\Omega = \{(x,y):x^2+y^2 <1\}$ in logarithmic scales and the numerical fitting curve $\hat{l}^*(p) = 3.69\times1.46^{-p}/(p-1)$ (right).}
  \label{sym_breaking_2dim}
\end{figure}

\section{Concluding remarks}\label{section:conclusion}
In this paper, for a generic functional in Hilbert space, we established the local variational characterization of its 1-saddles on the Nehari manifold and then transformed the computation of 1-saddles into a minimization problem on the Nehari manifold. Based on this, a framework of the Nehari manifold optimization method (NMOM) was constructed to compute the 1-saddles and the global convergence was established. The effectiveness of the proposed algorithm was successfully demonstrated through the computation of unstable ground states for both the stationary nonlinear Schr\"odinger equation and the H\'enon equation.

We remark that our work bridges Riemannian optimization and the computation of 1-saddles in Hilbert space, and thus, provides new perspectives for grasping unstable critical points. In addition, the developed weak convergence technique is expected to be extended to the convergence analysis for Riemann optimization methods on general infinite-dimensional manifolds. Nevertheless, the study is currently in the preliminary stage, and there are many directions for future investigations. Firstly, 1-saddles are characterized as local minimizers on the Nehari manifold in this paper, which cannot be applied directly to the general $k$-saddles ($k\geq2$). A new manifold should be introduced to allow a suitable variational characterization for $k$-saddles and then Riemannian optimization can be adopted to design corresponding numerical algorithms. Secondly, the NMOM is based on the line-search framework of Riemannian optimization and expected to be extended to some variants based on the trust-region framework in the future. Moreover, the analysis of the convergence rate for the NMOM is not provided in the current stage and can be considered in the future work. Finally, as a stable algorithm for finding 1-saddles, the NMOM is applicable to many nonlinear problems in various scientific fields, which will also be one of our future considerations.

\section*{Reproducibility of computational results}
 Source code that allow readers to reproduce the results in this paper are available at GitHub:\\
 \url{https://github.com/ChenZhaoXing-HUNNU/Nehari-manifold-optimization-method}.
 
\section*{Acknowledgments}
The authors would like to express their appreciation to the anonymous reviewers for the invaluable comments that have greatly improved the quality of the manuscript. Part of the work was done when W.~Liu was employed by National University of Singapore.

\appendix

\section{Proof of Lemma~\ref{lem:E0bounds}}\label{sec:appendix-proof-lem:E0bounds}
The Sobolev embedding $H\hookrightarrow L^{p+1}(\Omega)$, the expression \eqref{manifold} and $g\in L^\infty(\Omega)$ imply that
	\begin{align*}
		\|u\|_{L^{p+1}}^2 \leq C\|u\|_H^2
		 = C\int_{\Omega}g|u|^{p+1}d\mathbf{x}\leq C\|g\|_{L^\infty} \|u\|_{L^{p+1}}^{p+1},\quad\forall\;u\in\mathcal{N},
	\end{align*}
where $C>0$ is a constant independent of $u$. Then $\|u\|_{L^{p+1}} \geq \left(C\|g\|_{L^{\infty}}\right)^{-1/(p-1)} =: \tilde{\sigma}$. The proof is completed immediately by setting $\sigma_0 = \min\{\tilde{\sigma}, \tilde{\sigma}/\sqrt{C}, \tilde{\sigma}^2/C \}$.

\section{Proof of Lemma~\ref{lem:5.2}}\label{sec:appendix-proof-lem:5.2}
Recalling \eqref{eq:Rgrad=ProjHgrad}, one has
	\begin{align}\label{eq:DEun-DNEun}
     \nabla E(u_n) = \nabla_{\mathcal{N}} E(u_n) + \frac{(\nabla E(u_n), \nabla G(u_n))_H}{\|\nabla G(u_n)\|_H^2}\nabla G(u_n).
    \end{align}
Taking $H$ inner products with $u_n$ on both sides, one can conclude that
	\begin{align}\label{eq:DEun-DNEun-un} 
 (\nabla E(u_n),u_n)_H=(\nabla_{\mathcal{N}} E(u_n), u_n)_H + \frac{(\nabla E(u_n), \nabla G(u_n))_H}{\|\nabla G(u_n)\|_H^2}(\nabla G(u_n),u_n)_H.
\end{align}
Since $u_n\in\mathcal{N}$, there holds $(\nabla E(u_n),u_n)_H = 0$ and $E(u_n) = \big(\frac{1}{2} - \frac{1}{p+1}\big)\|u_n\|_H^2$. The boundedness of $E(u_n)$ shows that $\|u_n\|_H$ is bounded in $H$ and thus 
\begin{align}\label{eq:R_E and u converges to 0}
     (\nabla_{\mathcal{N}}E(u_n),u_n)_H \to 0,\quad n\to \infty.
\end{align}
The H\"older inequality and Sobolev embedding $H\hookrightarrow L^{p+1}(\Omega)$ imply that
\begin{align*}
\|\nabla G( u_n)\|_H 
=\sup_{0\neq\phi\in H}\frac{\langle G'(u_n),\phi\rangle}{\|\phi\|_H}
 &= \sup_{0\neq\phi\in H}\frac{2(u_n,\phi)_H - (p+1)\int_{\Omega} g|u_n|^{p-1}u_n \phi d\mathbf{x} }{\|\phi\|_H} \\
 &\leq  2\|u_n\|_H + (p+1)C \|g\|_{L^{\infty}} \|u_n\|_H^p,
\end{align*}
where $C>0$ is a constant independent of $u_n$. Then $\{\nabla G(u_n)\}$ is also bounded in $H$. In addition, Lemma~\ref{lem:E0bounds} states that $\{\|u_n\|_H\}$ has a positive lower bounded, then
$(\nabla G(u_n),u_n) = 2\|u_n\|_H^2 - (p+1)\int_{\Omega} g|u_n|^{p+1} d\mathbf{x} = (1-p)\|u_n\|_H^2$  has negative upper bound. Hence
$\|\nabla G(u_n)\|_H \geq |(\nabla G(u_n),u_n)_H| / \|u_n\|_H$ has a positive lower bound. As a result,
\begin{align}\label{eq:nabla G noconverges_to_0}
 \|\nabla G(u_n)\|_H \nrightarrow \infty, \quad \|\nabla G(u_n)\|_H  \nrightarrow 0 ,\quad
 (\nabla G(u_n), u_n)_H \nrightarrow  0.
\end{align}
With \eqref{eq:R_E and u converges to 0} and \eqref{eq:nabla G noconverges_to_0}, letting $n\to \infty$ in \eqref{eq:DEun-DNEun-un} leads to $(\nabla E(u_n),\nabla G(u_n))_H \to 0$. The proof is completed by letting $n\to \infty$ in \eqref{eq:DEun-DNEun}.

\section{Proof of Lemma~\ref{lem:5.3}}\label{sec:appendix-proof-lem:5.3}
The fact $E(u) = \big(\frac{1}{2} - \frac{1}{p+1}\big)\|u\|_H^2, u\in \mathcal{N}$ with $p>1$ shows that $\mathcal{E}_{u_0}$ is bounded in $H$. Denote $B_{\delta}(0)=\{v\in H:\|v\|_H<\delta\}$. By Lemma~\ref{lem:E0bounds} and the continuity of $\int_{\Omega} g|u|^{p+1}d\mathbf{x}$ w.r.t. the perturbations of $u$, we know that there exist $\tilde{c_1},\tilde{c_2}>0$ and $\delta>0$ s.t., for all $\xi\in T_u\mathcal{N}\cap B_{\delta}(0)$,
\begin{align}\label{bound of u and xi}
		\tilde{c_1} \leq\|u+\xi\|_H\leq \tilde{c_2},  \quad 
		\tilde{c_1} \leq\|u+\xi\|_{L^{p+1}}\leq \tilde{c_2}, \quad
		\tilde{c_1} \leq\int_{\Omega} g|u+\xi|^{p+1}d\mathbf{x}\leq \tilde{c_2},
\end{align} 
and therefore, $\rho(u+\xi)=\big(\|u+\xi\|_H^2/\int_{\Omega }g|u+\xi|^{p+1}d\mathbf{x}\big)^{\frac{1}{p-1}}$ is uniformly bounded from above and below by positive constants. The directional derivative of $\rho$ at $u+\xi$ w.r.t. a direction $\eta\in H$ reads
\begin{align*}
&\langle\rho'(u+\xi),\eta \rangle \\
&\quad\;\; = \frac{2\|u+\xi\|_H^{\frac{4-2p}{p-1}}(u+\xi,\eta)_H}{(p-1)\left(\int_{\Omega }g|u+\xi|^{p+1}d\mathbf{x}\right)^{\frac{1}{p-1}}}  
 - \frac{(p+1)\|u+\xi\|_H^{\frac{2}{p-1}}\int_{\Omega }g|u+\xi|^{p-1}(u+\xi)\eta d\mathbf{x}}{(p-1)\left(\int_{\Omega }g|u+\xi|^{p+1}d\mathbf{x}\right)^{\frac{p}{p-1}}}.
\end{align*}
It follows from \eqref{bound of u and xi}, H{\"o}lder inequality and Sobolev embedding $H\hookrightarrow L^{p+1}(\Omega)$ that there exists a constant $C_L$ independent of $u$ and $\xi$ s.t. $\left|\langle \rho'(u+\xi),\eta \rangle \right| \leq C_L \|\eta\|_H$, $\forall\;\eta\in H$, and thus
% \vspace{-0.5ex}
\begin{align}\label{the derivation of rho is bounded}
\|\rho'(u+\xi)\|_{H^*}\leq C_L, \quad \forall \;u\in \mathcal{E}_{u_0},\ \forall \;\xi \in T_u\mathcal{N}\cap B_{\delta}(0).
\end{align}
Noticing $\mathcal{R}_u(\xi) - u = (\rho(u+\xi)-\rho(u))u + \rho(u+\xi)\xi=\langle \rho'(u+\theta\xi),\xi\rangle u + \rho(u+\xi)\xi$ for some $\theta \in (0,1)$, the condition \eqref{lipschitz continuous} can be verified immediately by utilizing \eqref{the derivation of rho is bounded} and the boundedness of $u$ and $\rho(u+\xi)$. The proof is finished.
% \vspace{-1.5ex}

\bibliographystyle{siamplain}
\bibliography{references}

\end{document}